\documentclass[a4paper]{amsart}

\newtheorem{theorem}{Theorem}[section]
\newtheorem{lemma}[theorem]{Lemma}
\newtheorem{proposition}[theorem]{Proposition}

\newtheorem{corollary}[theorem]{Corollary}
\usepackage{dsfont, comment, bm}

\theoremstyle{definition}
\newtheorem{definition}[theorem]{Definition}

\theoremstyle{remark}
\newtheorem{remark}[theorem]{Remark}

\numberwithin{equation}{section}

\newcommand{\I}{{\mathds {1}}}
\newcommand{\M}{\mathcal{M}}
\newcommand{\N}{\mathcal{N}}
\newcommand{\tM}{\widetilde{\mathcal{M}}}
\newcommand{\decra}{\mu}

\begin{document}

\title{A noncommutative transfer principle}

\author{Louis Labuschagne}
\address{DSI-NRF CoE in Math. and Stat. Sci, Pure and Applied Analytics\\ NWU, PVT. BAG X6001, 2520 Potchefstroom, South Africa}
\email{louis.labuschagne@nwu.ac.za}
\author{Claud Steyn}
\address{DSI-NRF CoE in Math. and Stat. Sci, Pure and Applied Analytics\\ Internal Box 209, School of Math. \& Stat. Sci.,
NWU, PVT. BAG X6001, 2520 Potchefstroom, South Africa}
\email{claud.steyn@nwu.ac.za}
\subjclass[2020]{Primary 37A15, 37A30, 46L52, 46L55; Secondary 22F10, 37A46, 37A55, 46E30}
\date{\today}
\thanks{}
\begin{abstract}  We extend the generalised Calderon Transfer Principle as presented in \cite{dBL} to the setting of trace preserving group actions of $\sigma$-compact 
locally compact Hausdorff groups on semifinite von Neumann algebras $\M$. At its essence the transfer principle consists of showing that a large class of regular operators 
in the group context may be translated to operators in the algebra context by means of this group action. The end result is a protocol for proving ergodic convergence 
results for group actions on noncommutative Orlicz space of semifinite von Neumann algebras. We start with proving the existence of the transferred operator in the 
noncommutative context. Two tools are developed for the purpose of actually proving convergence results: (1) a theory of what may be called 2-variable decreasing 
rearrangements for the algebra $L^\infty \overline{\otimes} \M$ (where ($\Gamma,\nu$) is a Radon measure space) and (2) a concepts of maximal operators suited to the present 
context. These tools are then used to lift the ergodic convergence results presented in \cite{dBL} to the noncommutative setting. In closing we present examples illustrating 
the application of the tools 
developed.\end{abstract}

\maketitle

\tableofcontents

\section{Preliminaries}

Calderon first conceived of the transfer principle in his iconic 1968 paper \cite{Cal} dealing with ergodic theorems for actions of the real numbers on $\sigma$-finite 
measure spaces. Since then much has been written about this topic. Amongst these the genesis of the current paper may however be found in the 2016 paper of de Beer and the 
first named author \cite{dBL} in which the transfer principle was formalised and its efficacy for integrable actions of locally compact abelian groups on measure spaces 
demonstrated. This paper in particular demonstrated the efficacy of this principle for obtaining pointwise ergodic theorems for Orlicz spaces. Then in 2024 the first-named 
author reworked the de Beer-Labuschagne \cite{dBL} approach to produce a revised approach more suited to possible future extension to quantum measure theory and Orlicz 
spaces (see \cite{L-QM}). The actual extension of this theory to the quantum setting is then the goal of this paper. In our study we shall be working with locally compact 
groups $G$ equipped with a left-invariant Haar measure and write $e$ for the group unit. For the sake of access to Theorem \ref{Mal-nc}, which in turn relies on 
\cite[Theorem 3]{Mal}, we require our groups to be $\sigma$-compact (ensuring $\sigma$-finiteness of the Haar measure). This is however a very mild and prefectly acceptable 
restriction. Unless otherwise stated we shall assume the groups we are working with to be multiplicative. Given a group action on some measure space and a dynamical process in 
the group context described by some sequence $(T_n)$ of operators, the basic goal of the transfer principle is to use the group action as a tool for transferring the operators 
$(T_n)$ to a matching sequence of operators $(T^\#_n)$ in the measure space context. Once this is a done, a form of the Banach principle is used to extract convergence 
properties of the sequence $(T^\#_n)$ from maximal inequalities inherited from the sequence $(T_n)$, the convergence of which then ultimately enables us to obtain ergodic 
convergence theorems.

The established formalism for the description and study of noncommutative spaces of measurable functions are von Neumann algebras. Von Neumann algebras $\M$ are in essence 
weak operator closed subalgebras of the space $B(H)$ of bounded operators on a Hilbert space which are closed under the Hilbert adjoint. We hasten to add that von Neumann 
algebras can also be defined abstractly as the class the class of $C^*$-algebras that are also Banach dual spaces in which case the term $W^*$-algebras is preferred. 
However ultimately the von Neumann algebras and $W^*$-algebras (abstract von Neumann algebras) turn out to be equivalent.  Each von Neumann algebra can be equipped with an 
affine map $\nu$ from the cone $\M_+$ of positive elements to the extended half-line $[0,\infty]$ which is 
\begin{itemize}
\item \emph{faithful} in the sense that for any $f\in\M_+$ we have $\nu(f)=0$ if and only if $f=0$;
\item \emph{normal} in the sense that if $(f_\alpha)\subset \M_+$ increases monotonically to $f$, then $\nu(f_\alpha)\nearrow\nu(f)$;
\item and \emph{semifinite} if the set $\{a\in \M\colon \nu(a^*a)<\infty\}$ is $\sigma$-weakly dense in $\M$. 
\end{itemize}
These faithful semifinite normal weights (or just fns for short) play the role of a noncommutative integral. Weights satisfying the additional regularity constraint that 
$\nu(a^*a)=\nu(aa^*)$ for each $a\in \M$ are called traces and typically denoted by $\tau$. In fact the class of semifinite von Neumann algebras are characterised by the 
existence of precisely such fns traces. The identification of von Neumann algebras as the appropriate formalism for noncommutative measure and integration theory is based on 
the fact that the category of commutative von Neumann algebras equipped with a faithful normal semifinite trace is equivalent to the category of localizable measure algebras 
(see \cite[\S 1.8]{GLbook}).  We shall here restrict attention to semifinite von Neumann algebras equipped with faithful normal semifinite traces. In the construction of the 
transferred operator $T^\#$ we shall make use of the projective tensor product of $L^1$ spaces with Banach spaces and also tensor products of von Neumann algebras with a 
specific focus on the case when one of the factors is a commutative von Neumann algebra. We shall assume that readers are familiar with the essentials of the projective and 
injective tensor product. As far as von Neumann algebras are concerned we point out that the family of commutative von Neumann algebras equipped with an fns trace are also 
equivalent to the family of $L^\infty(\Gamma,\nu)$ spaces where $(\Gamma,\Sigma,\nu)$ is a Radon measure spaces (see \cite[Theorem 1.54]{GLbook}). On this point we note that 
a locally compact group $G$ equipped with left-invariant Haar measure is in fact a Radon measure space. Takesaki's exposition of tensor products of $W^*$-algebras with such 
$L^\infty(\Gamma,\nu)$ spaces (see \cite[\S IV.7]{Tak} is therefore directly relevant for us. We shall expound the salient details of these tensor products below.

As mentioned our goal is to obtain a theory which is efficacious for not just $L^p$, but also Orlicz spaces. The description of these and more general spaces requires some 
knowledge of the theory of decreasing rearrangements for semifinite von Neumann algebras. We here give a very brief overview. For a fuller introduction please refer to 
\cite[Chapters 5-7]{GLbook}. The first fact one needs to note is that any concrete von Neumann algebra $\M\subset B(H)$ equipped with a faithful semifinite trace 
$\tau$ may be enlarged to the so-called set of closed densely defined affiliated operators. A densely defined closed operator $a$ on a Hilbert space $H$ is said to be 
affiliated to $\M$ (denoted by $a\eta\M$) if it commutes with every unitary in the commutant of $\M$, or equivalently if in the polar decomposition $a=u|a|$, $u$ as well as 
each spectral projection of $|a|$ belongs to $\M$. We are however particularly interested in those affiliated operators $a$ for which the spectral projection 
$\chi_{(\lambda,\infty]}(|a|)$ will for some $\lambda>0$ have finite trace. Under strong sum and product this class turns out to be a complete metrisable algebra known as 
algebra of $\tau$-measurable operators and denoted by $\tM$. 

We will write $\mathbb{R}_+$ for the closed half-line and denote the associated Lebesgue measure by $m$. For each densely defined closed operator $a$ affiliated to $\M$, one 
may define an associated distribution function $\mathbb{R}_+\to [0,\infty]:t\mapsto d_a(t)$ by the prescription $d_a(t)=\tau(\chi_{(t,\infty)}(|a|))$. The algebra $\tM$ then 
consists of all densely defined closed affiliated operators $a$ for which the associated distribution is ultimately finite-valued. This algebra is then equipped with a 
tracial topology of ``convergence in measure'' under which it proves to be a complete metric *-algebra. (See \cite[\S 3.5]{GLbook} for details.) The classical counterpart 
of this algebras is just the space $M^0(\Omega,\Sigma,\nu)$ of all almost everywhere finite measurable functions equipped with the topology of convergence in measure (see 
\cite[Exercise 1.1]{BS}).  

For each element $a$ of $\tM$ one may then define an associated ``decreasing rearrangement'' $\decra(\cdot,a):\mathbb{R}_+\to \mathbb{R}_+:t\mapsto \decra(t,a)$ by the 
prescription $\decra(t,a)=\inf\{s\geq 0\colon d_a(s)\leq t\}$. (Here $a.e$ denotes the strong product.) There are other ways of defining the decreasing rearrangement, but 
this one is most convenient for us as it clearly shows that this definition is a direct translation of the classical decreasing rearrangement $f^*$ of a function 
$f\in L^\infty(\Omega,\Sigma,\nu)$ which is defined by the prescription $f^*(t)= \inf\{s\geq 0\colon \nu(\{w\in\Omega\colon |f(w)|>s\}), \leq t\}$. This noncommutative 
decreasing rearrangement then inherits much of the elegance of its classical counterpart which is what makes the theory of noncommutative rearrangement invariant Banach 
function spaces possible. For details on Banach function spaces and decreasing rearrangements in the classical context, please refer to \cite{BS, KPS}. The approach of 
\cite{KPS} is slightly more general than that of \cite{BS}. We shall for now however follow the approach of \cite{BS}. For our purposes it is enough to restrict ourselves to 
the measure space $(\mathbb{R}_+,m)$. The description below is a truncated and simplified version of a similar description given in \cite{L-QM}. We repeat it for the 
sake of completeness. In the Bennett-Sharpley approach \cite[Definition 1.1.1 \& Theorem 1.1.4]{BS} a (Banach) function norm $\bm{\rho}$ on the cone $M^0_+(\mathbb{R}_+,m)$ 
of non-negative almost everywhere finite measurable functions is defined to be a mapping $\bm{\rho} : M^0_+ \to [0, \infty]$ satisfying
\begin{itemize}
\item $\bm{\rho}(f) = 0$ iff $f = 0$ a.e.  
\item $\bm{\rho}(c f) = c\bm{\rho}(f)$ for all $f \in M^0_+, c > 0$.
\item $\bm{\rho}(f + g) \leq \bm{\rho}(f) + \bm{\rho}(g)$ for all $f$, $g$.
\item $f \leq g$ implies $\bm{\rho}(f) \leq \bm{\rho}(g)$ for all $f, g \in M^0_+$, with $\bm{\rho}(f_n)\nearrow\bm{\rho}(f)$ 
if $f_n \nearrow f$ $m$-a.e.
\item $m(E)<\infty$ implies $\bm{\rho}(\chi_E)<\infty$.
\item For any $E\in \Sigma$ with $m(E)<\infty$ there exists $0<C_{E,\bm{\rho}}<\infty$ such that $\int_Ef\,dm\leq C_{E,\bm{\rho}}\bm{\rho}(f)$.
\end{itemize}
Such a $\bm{\rho}$ may be extended to all of $M^0(\mathbb{R}_+,m)$ by setting $\bm{\rho}(f) = \bm{\rho}(|f|)$, in which case we may then define $L^{\bm{\rho}}(\mathbb{R}_+)$ 
to be the space $L^{\bm{\rho}}(\mathbb{R}_+) = \{f \in M^0(\mathbb{R}_+,m) : \bm{\rho}(f) < \infty\}$ equipped with the norm $\|f\|_{\bm{\rho}}= \bm{\rho}(|f|)$. We shall 
refer to this space as the Banach Function Space corresponding to $\bm{\rho}$. These spaces all satisfy a version of Fatou's lemma, which in turn ensures their completeness 
\cite[Lemma 1.1.5]{BS}. 

Two measurable functions $f,g \in M^0(\mathbb{R}_+,m)$ are said to be equimeasurable if their distribution functions agree. A Banach function norm $\bm{\rho}$ is in turn said 
to be rearrangement invariant if $\bm{\rho}(f)=\bm{\rho}(g)$ whenever $f$ and $g$ are equimeasurable. A Banach function space produced by a rearrangement invariant norm is 
similarly called rearrangement invariant. The Luxemburg representation theorem \cite[Theorem 2.4.10]{BS} informs us that any rearrangement invariant Banach function space 
produced by some rearrangement invariant Banach function norm $\bm{\rho}$ on any $\sigma$-finite measure space $(\Omega,\Sigma,\nu)$ which is either non-atomic or fully 
atomic with all atoms having equal measure, may be reconstructed using decreasing rearrangements of the elements of this space and a matching Banach function norm 
$\widetilde{\bm{\rho}}$ on $M^0(\mathbb{R}_+,m)$. Algebraically the space turns out to be exactly the set $\{f\in M^0(\Omega,\nu)\colon f^*\in 
L^{\widetilde{\bm{\rho}}}(\mathbb{R}_+)\}$ with the norm given by $\bm{\rho}(f)=\widetilde{\bm{\rho}}(f^*)$. In their 1989 paper \cite{DDdP} Dodds, Dodds and de Pagter showed 
that for any rearrangement invariant Banach function norm $\bm{\rho}$ of the above type, the set $\{a\in\tM\colon \decra(\cdot,a)\in L^{\bm{\rho}}\mathbb{R}_+\}$ is a linear 
space for which $\|a\|_{\bm{\rho}} = \bm{\rho}(\decra(\cdot,a))$ is a norm under which this space is in fact a Banach space. On the basis of the intuition gained from the 
Luxemburg representation theorem we may justifiably regard such spaces as noncommutative Banach function spaces. Remarkably, in his 1937 paper on norms on matrix algebras 
\cite{vNA}, the work of John von Neumann already anticipated this theory of noncommutative Banach function spaces.

Every rearrangement Banach function space $L^{\bm{\rho}}(\mathbb{R}_+)$ admits a so-called K\"othe dual ${L^{\bm{\rho}}}'$ (also known as the 
associate space) which is defined to be the space ${L^{\bm{\rho}}}'(\mathbb{R}_+) =\{f\in M^0(\mathbb{R}_+)\colon fh\in L^1(\mathbb{R}_+)\mbox{ for all }g\in 
L^{\bm{\rho}}(\mathbb{R}_+)\}$ and which proves to be a rearrangement Banach function space when equipped with the norm $\bm{\rho}'$ on $M^0(\mathbb{R}_+)$ given by 
$\bm{\rho}'(f)=\sup\{\int |fg|\,dm\colon f\in M^0_+, \bm{\rho}(f)\leq 1\}$. The norm $\bm{\rho}'$ is itself also a rearrangement invariant Banach function norm. We also learn 
from \cite[Theorem 1.2.7]{BS} that Banach function spaces constructed from the fairly regular rearrangement invariant Banach function norms described above are all ``K\"othe 
reflexive''. We refer the reader to \cite[Subsection 6.3.1]{GLbook} for a brief resum\'e of K\"othe duality for the noncommutative variants of the above spaces. As was shown 
by Dodds, Dodds and de Pagter \cite{DDdP-Kothe} all of this theory carries over to the noncommutative context.   

A slightly more general but still very elegant class of rearrangement invariant Banach function spaces is the class produced by the so-called monotone Riesz-Fischer norms 
(see \cite[Definition 5.1.16 \& Theorem 5.1.17]{BS}, \cite[\S 7.2]{GLbook}) for details. These spaces characterise precisely the exact interpolation spaces of the pair 
$(L^\infty, L^1)$ in even the quantum world of semifinite von Neumann algebras. This fact then also shows that for any semifinite von Neumann algebra there is a largest and 
smallest rearrangement invariant Banach function space namely $(L^1+L^\infty)$ and $(L^1\cap L^\infty)$ respectively. We shall make regular use of these spaces and for the 
sake of convenience introduce the shorthand notation $(L^1+L^\infty)=L^{1+\infty}$ and $(L^1\cap L^\infty)=L^{1\cap\infty}$. All the spaces we are interested in are 
of the type defined earlier. The monotone Riesz-Fischer spaces are therefore merely mentioned for the sake of context. 

We proceed with describing the spaces we are primarily interested in. Here we shall again closely follow the corresponding discussion in \cite{L-QM}. Rearrangement invariant 
Banach function spaces $L^{\bm{\rho}}$ have the property that for any two finitely measurable sets $E_1$ and $E_2$ with the same measure, we have that 
$\|\chi_{E_1}\|_{\bm{\rho}}= \|\chi_{E_2}\|_{\bm{\rho}}$. This then enables one to introduce the notion of a \textit{fundamental function} for each such space. Given such a 
space $L^{\bm{\rho}}$, the corresponding fundamental function $\varphi_{\bm{\rho}}$ is given by the formula $$\varphi_{\bm{\rho}}(t)=\|\chi_E\|_X \mbox{ for all } E\in \Sigma 
\mbox{ with } \nu(E)=t$$where the domain of $\varphi_{\bm{\rho}}$ consists of all $t\in (\mathbb{R}_+)$ for which there is some measurable set $E$ with $\nu(E)=t$. Readers 
may find an excellent survey of fundamental functions in \cite[\S 2.5]{BS}. We mention a few of the essentials. One fact worthy of noting is that for any rearrangement 
invariant Banach function space the fundamental functions of the pair $(L^{\bm{\rho}}, (L^{\bm{\rho}})')$ satisfy the equality 
$\varphi_{\bm{\rho}}(t)\varphi_{\bm{\rho}'}(t)=t$ for all admissible values of $t\geq 0$ \cite[Theorem 2.5.2]{BS}. We may think of $\varphi_{\bm{\rho}'}$ as the dual 
fundamental function of $\varphi_{\bm{\rho}}$. It further follows that these functions correspond exactly to the class of so called \textit{quasi-concave} functions, namely 
functions $\varphi$ which are 0 only at 0, non-decreasing, and with 
$t\to \frac{\varphi(t)}{t}$ non-increasing. 

On the basis of the Luxemburg representation theorem and its quantum analogue, we may restrict ourselves to fundamental functions corresponding to Banach function spaces on 
$(\mathbb{R}_+,m)$ and hence assume all fundamental functions to be defined on all of $\mathbb{R}_+$. For future reference we also note that although fundamental 
functions may not actually be concave, any fundamental function is always equivalent to a concave fundamental function \cite[Proposition 2.5.10]{BS}. Any rearrangement 
invariant Banach function space on a $\sigma$-finite measure space which is either non-atomic, or fully atomic with all atoms having equal measure, may correspondingly be 
equivalently renormed in such a way that the fundamental function is in fact concave \cite[Proposition 2.5.10]{BS}. 

There may be several rearrangement invariant Banach function spaces corresponding to any specific quasi-concave function $\varphi$, but in the family of rearrangement 
invariant Banach function spaces described by such a function, there is for the measure space $(\mathbb{R}_+,m)$ always a largest and 
smallest space with $\varphi$ as fundamental function - the so-called maximal and minimal Lorentz spaces corresponding to $\varphi$ (see \cite[Corollary 2.5.14]{BS} and the 
comment following it). Assuming $\varphi$ to actually be concave, these spaces are respectively defined by $$M_\varphi(\mathbb{R}_+,m)= 
\{f\in M^0\colon \sup_{0<t<\infty}\frac{\varphi(t)}{t}\int_0^t f^*(s)\,ds<\infty\}$$and $$\Lambda_\varphi(\mathbb{R}_+,m)=\{f\in M^0\colon 
\int_0^\infty f^*(s)\,d\varphi(s)<\infty\}$$with the norms given by $\|f\|_{M_\varphi}= \sup_{0<t<\infty}\frac{\varphi(t)}{t}\int_0^tf^*(s)\,ds$ and $\|f\|_{\Lambda_\varphi} 
= \int_0^\infty f^*(s)\,d\varphi(s)$. Given a fundamental function $\varphi$ there is in fact a larger space we may construct from this function, namely 
$$M^*_\varphi(\mathbb{R}_+,m)= \{f\in M^0\colon \sup_{0<t<\infty}\varphi(t)f^*(t)<\infty\},$$but this space is in general at best a quasi-normed space (with the quasi-norm 
given by $q_\varphi(f)=\sup_{0<t<\infty}\varphi(t)f^*(t)$).

In the family of rearrangement invariant Banach function spaces we are particularly interested in Orlicz spaces. Each Orlicz space is defined in terms of a \emph{Young 
function} by which term we mean a convex function $A : [0, \infty) \to [0, \infty]$ satisfying $A(0) = 0$ and $\lim_{u \to \infty} A(u) = \infty$, which is neither 
identically zero nor infinite valued on all of $(0, \infty)$, and which is left continuous at $b_A = \sup\{u > 0 : A(u) < \infty\}$. We will hereafter consistently write $A$, 
$B$, $C$, and $X$, $Y$, $Z$ for Young functions. It is worth pointing out that any Young function must also be increasing, and continuous on $[0, b_A]$. The right-continuous 
inverse $A^{-1}: [0, \infty) \to [0, \infty]$ of a Young function is given by the formula $$A^{-1}(t) = \sup\{s : A(s) \leq t\}.$$

There are various growth conditions by means of which Young functions are classified in the literature. One very common growth condition which will be relevant for us, is the 
so-called $\Delta_2$ condition. A Young function $A$ is said to \emph{satisfy $\Delta_2$ globally} if there exists a constant $c>0$ such that $A(2s)\leq cA(s)$ for all 
$s\geq 0$. We further note that each Young function $A$ induces a complementary Young function $A^*$ which is defined by $A^*(u) = \sup_{v > 0}(uv - A(v))$. The pair 
$(A,A^*)$ satisfy the following Hausdorff-Young inequality: $$st \leq A(s) + A^*(t) \qquad s,t \geq 0.$$

The Orlicz space $L^A(\mathbb{R}_+,m)$ associated with $A$ is defined to be the set $$L^{A} = \{f \in M^0(\mathbb{R}_+,m) : A(\lambda |f|) \in L^1 \quad 
\mbox{for some} \quad \lambda = \lambda(f) > 0\}.$$This space turns out to be a linear subspace of $M^0(\mathbb{R}_+,m)$ which proves to be a Banach space when 
equipped with the so-called \emph{Luxemburg-Nakano norm} $$\|f\|_A = \inf\{\lambda > 0 : \|A(|f|/\lambda)\|_1 \leq 1\}.$$An equivalent norm (the \emph{Orlicz norm}) is given 
by the formula $$\|f\|^O_A = \sup\{|\textstyle{\int} fg\, dm| : g\in L^{A^*}, \|g\|_{A^*}\leq 1\}.$$As Banach function spaces the space $(L^{A^*},\|\cdot\|^O_{A^*})$ 
is precisely the K\"othe dual of $(L^A,\|\cdot\|_A)$ \cite[Corollary 4.8.15]{BS}. From our discussion of Banach function spaces it therefore follows that the fundamental 
function of the space $(L^{A^*},\|\cdot\|^O_{A^*})$, is nothing but $\varphi_A'$. The Lebesgue $L^p$ spaces (with $1\leq p<\infty$) are of course just Orlicz spaces 
corresponding to the Young function $\psi(t)=t^p$. Although the notational conventions for Banach function spaces, Orlicz spaces and Lebesgue spaces do differ somewhat, we 
will regardless abide by well established conventions trusting that in each case the context will clarify which spaces are in view.

For Orlicz spaces $L^A$ equipped with the Luxemburg norm, one has access to a very elegant formula for the fundamental function of $L^A$, namely the fact that  
\begin{equation}\label{fundlux}
\varphi_A(t)=\frac{1}{A^{-1}(1/t)}
\end{equation}
(see \cite[Lemma 4.8.17]{BS}.) Thus the fundamental function can be constructed from the given Young function. However more is true. Correspondingly in the case where 
$\varphi$ is a concave fundamental function, it is an exercise to see that $A(t)=\frac{1}{\varphi^{-1}(1/t)}$ then defines a Young function for which 
$\frac{1}{A^{-1}(1/t)}=\varphi(t)$. So up to equivalence with concave fundamental functions, there is an almost bijective correspondence between fundamental functions and 
Young functions. When dealing with an Orlicz space $L^A$, we may therefore justifiably, where convenient, denote the smallest and largest rearrangement invariant spaces 
corresponding to $\varphi_A$ by $\Lambda_A$ and $M_A$.

Given a pair of Young functions $(A,B)$, we shall say that a linear operator $T$ defined on $L^{1+\infty}(\mathbb{R}_+,m)$ is of \emph{strong type $(A,B)$} if it 
restricts to a bounded operator from $L^A$ to $L^B$ and of \emph{weak type $(A,B)$} if instead it restricts to a bounded operator from $L^A$ to $M_B$. Some authors define 
weak type in terms $M^*_B$ instead of $M_B$, but for technical reasons we shall prefer $M_B$. 

In our analysis we shall have occasion to make use of Orlicz-Bochner spaces and also mixed norm spaces. For the sake of the reader we repeate a brief survey of the 
essentials as presented in \cite{L-QM}. Given a measure space $(\Omega,\Sigma,\mu)$ and a Banach space $\mathcal{X}$, a vector valued function $F:\Omega\to \mathcal{X}$ is 
said to be \emph{Borel measurable} if it is measurable with respect to $\Sigma$ and the Borel $\sigma$-algebra of $\mathcal{X}$. The function $F$ is said to be {strongly 
measurable} if it is both Borel measurable and separably valued. It is not difficult to see that if $F$ is strongly measurable, then $t\to \|F_t\|_{\mathcal{X}}$ is a 
$\mu$-measurable real-valued function. (See the introduction to Appendix E of \cite{Cohn}.) For a given Young function $A$, the space $L^A((\Omega,\Sigma,\mu),\mathcal{X})$ 
is defined to be the space of all (equivalence classes of) strongly measurable functions $F$ from $\Omega$ to $\mathcal{X}$ for which the function $t\to 
\|F_t\|_{\mathcal{X}}$ belongs to $L^A(\Omega,\Sigma,\mu)$, with the norm given by $\|F\|_{A,\mathcal{X}} = \|\,\|F(\cdot)\|_{\mathcal{X}}\,\|_A$. It is well-known and very 
useful fact that in the specific case of $L^1$, the tensor product $L^1(\Omega,\Sigma,\mu)\otimes_\pi E$ is for any Banach space $E$ an isometric copy of the Bochner space 
$L^1((\Omega,\Sigma,\mu);E)$.

A closely related concept, is that of a \emph{mixed norm space}. For any two rearrangement invariant Banach function spaces $\mathcal{X}$ and $\mathcal{Y}$ on resonant 
measure spaces $(\Omega_1,\Sigma_1, \mu_1)$ and $(\Omega_2,\Sigma_2, \mu_2)$ respectively, the mixed norm space $\mathcal{X}[\mathcal{Y}]$ (in the sense of Benedek and 
Panzone \cite{BP}) is defined to be the space of all $\mu_1\times\mu_2$-measurable functions $(x,y)\to f(x,y)$ on $\Omega_1\times\Omega_2$ for which 
$x\to \|f_x\|_{\mathcal{Y}}$ belongs to $\mathcal{X}$. (Here $f_x$ is for any fixed $x$ of course the function $y\to f(x,y)$. The canonical norm on such a mixed norm space 
$\mathcal{X}[\mathcal{Y}]$ is just $\|f\|_{\mathcal{X}[\mathcal{Y}]} = \|\,\|f_x\|_{\mathcal{Y}}\,\|_{\mathcal{X}}$. The reader will find a very useful discussion of such 
spaces in \cite{Blo}. In the case where $\mathcal{X}=L^A$ for some Young function, the Bochner space $L^A(\Omega_1,\mathcal{Y})$ may also be regarded as a space of 
$\mu_1\times\mu_2$-measurable functions $(x,y)\to f(x,y)$ where each vector-valued function $F\in L^A(\Omega_1,\mathcal{Y})$ corresponds to the function $f_F: 
\Omega_1\times\Omega_2$ given by $f_F\colon(x,y)\to [F(x)](y)$. With $L^A(\Omega_1,\mathcal{Y})$ represented in this form, the $\mu_1$-measurability of 
$x\to \|F(x)\|_{\mathcal{Y}}$ then ensures that $L^A(\Omega_1,\mathcal{Y})$ is isometrically a subspace of $L^A[\mathcal{Y}]$. 

\section{Introduction}
\subsection{A brief synopsis of the algebra $L^\infty(\Gamma)\overline{\otimes}\M$}\label{vNA_MG}

Given a Radon measure space $(\Gamma,\nu)$, we shall need some insight into the structure of the von Neumann algebra $L^\infty(\Gamma)\overline{\otimes}\M$ and its predual as 
background for the construction of the quantum transfer operator. (We remind the reader that $L^\infty(\Gamma)$ may be represented as a concrete von Neumann algebra by 
identifying it with algebra of multiplication operators on $L^2(\Gamma)$ with symbols in $L^\infty(\Gamma)$.)  We therefore pause to review the theory of the von Neumann 
algebra $L^\infty(\Gamma)\overline{\otimes}\M$ . Regarding this point it should be noted that not only is any commutative von Neumann algebra a *-isomorphic copy of some 
$L^\infty(\Gamma)$, but that any locally compact group equipped with left-invariant Haar measure is itself of course a Radon measure space \cite[Chapter 1]{GLbook}. 

By \cite[Theorem IV.7.17]{Tak} the elements of $L^\infty(\Gamma)\overline{\otimes}\M$ are $\M$-valued bounded 
functions on $\Gamma$ which are measurable in the sense of \cite[Definition IV.7.7]{Tak} and is normed by the norm $\|{a}\|=\sup_{s\in \Gamma}\|a(s)\|$. By 
\cite[Corollary IV.7.8]{Tak} the *-algebra operations are defined in a pointwise fashion, specifically
$$({a}{b})(s)={a}(s){a}(s)\quad({a}+{b})(s)={a}(s)+{a}(s), ({a}^*)(s)={a}(s)^* \mbox{ for all }s\in \Gamma.$$It is clear from the above that projections ${p}$ in 
$L^\infty(\Gamma)\overline{\otimes}\M$ are those $\M$-valued bounded measurable functions on $\Gamma$ for which each ${p}(s)$ ($s\in \Gamma$) is again a projection in $\M$. 

It now follows from the above that we will for any self-adjoint element ${a}\in L^\infty(\Gamma)\overline{\otimes}\M$ and any polynomial $q$ have that $q({a})$ is a 
vector-valued function $\Gamma\to \M$ with $q({a})(s) = q({a}(s))$ for each $s\in \Gamma$. Using the Stone-Weierstrass theorem, one may deduce from this fact that this claim 
still holds true for continuous functions on the spectrum of $a$. 

\subsubsection{The bounded Borel functional calculus of $L^\infty(\Gamma)\overline{\otimes}\M$}

With the above as starting point, we will in turn briefly describe the structure of a projection-valued measure of any given self-adjoint operator affiliated to 
$L^\infty(\Gamma)\overline{\otimes}\M$ and the Borel functional calculus for such operators. For ease of notation we will write $\N$ for 
$L^\infty(\Gamma)\overline{\otimes}\M$ where convenient. Projection-valued measures and the Borel functional calculus is well-known territory for any operator algebraist, and 
hence, rather than provide full details, we shall merely outline how existing proofs of these facts may be adapted to the present context. First notice that for any 
$\alpha>0$ the sequence $(f_n^{(\alpha)})$ of continuous functions on $\mathbb{R}$ given by 
$$f_n^{(\alpha)}(t) =\left\{\begin{array}{ll} 0 &\quad 0\leq t\leq \alpha\\ (t-\alpha)n &\quad \alpha\leq t\leq \alpha+\tfrac{1}{n}\\ 1 &\quad 
t> \alpha+\tfrac{1}{n}\end{array}\right.$$
increases to the characteristic function $\chi_{(\alpha,\infty)}$. It follows from what was noted in the previous paragraph, thet we will for any self-adjoint $a\in \N$ and 
$n$ have that $$f_n^{(\alpha)}(a)(t) = f_n^{(\alpha)}(a(t))\mbox{ for all }t\in \Gamma.$$Letting $n\nearrow\infty$ then yields the fact that 
$$\chi_{(\alpha,\infty)}(a)(t) = \chi_{(\alpha,\infty)}(a(t))\mbox{ for all }t\in \Gamma.$$An obvious modification of \cite[1.6.1.1-1.6.1.3]{Bla} now yields the following 

\begin{proposition} Let $a$ be a self-adjoint element of $L^\infty(\Gamma)\overline{\otimes}\M$. For any Borel set $B\subset \mathbb{R}$ and any $s\in \Gamma$ the operator 
$\chi_B(a)(s)\in L^\infty(\Gamma)\overline{\otimes}\M$ is the spectral projection $\chi_B(a(s))$ of $a(s)\in \M$.
\end{proposition}

We shall ultimately reconfirm the above fact using different methods (see Proposition \ref{ptwise-spectral}). We note from \cite[Proposition IV.7.14]{Tak} that if 
$\M\subset B(H)$, then $L^\infty(\Gamma)\overline{\otimes}\M\subset B(L^2(\Gamma;H))$. With the above proposition as basis, technical modifications of 
Theorems 5.2.4-5.2.6 of \cite{KR} then lead to a ``point-wise'' version of \cite[Corollary 5.2.7]{KR}, namely the fact that for any $f\in C(\{\mathbb{R},\infty\})$ 
vanishing at infinity, any self-adjoint $a\in \N$ and any $x\in L^2(\Gamma;H)$ we have that 
\begin{eqnarray*}
\int_\Gamma \langle f(a(s))x(s),x(s)\rangle\,ds &=& \langle f(a)x,x\rangle\\
&=& \int_\mathbb{R}f(\gamma)\,d\langle e_\gamma x,x\rangle\\ 
&=& \int_\Gamma \int_\mathbb{R}f(\gamma)\,d\langle e_\gamma x(s),x(s)\rangle\,ds
\end{eqnarray*} 
where $e_\gamma$ and $e_\gamma(s)$ respectively denote the spectral projections $\chi_{(-\infty,\gamma]}(a)$ and $\chi_{(-\infty,\gamma]}(a(s))$. (As in \cite{KR}, 
$\{\mathbb{R},\infty\}$ here denotes the one-point compactification of $\mathbb{R}$.) With this technology at our disposal it is now a straightforward exercise to modify the 
discussion preceding \cite[Theorem 5.2.8]{KR} so as to be able to conclude that we have the following 

\begin{proposition}\label{Borel-bd} Let $a$ be a self-adjoint element of $L^\infty(\Gamma)\overline{\otimes}\M$. For any bounded Borel function $f$ on $\mathbb{R}$ and any 
$x\in L^2(\Gamma;H)$ we have that 
\begin{eqnarray*}
\int_\Gamma \langle f(a(s))x(s),x(s)\rangle\,ds &=& \langle f(a)x,x\rangle\\
&=& \int_\mathbb{R}f(\gamma)\,d\langle e_\gamma x,x\rangle\\ 
&=& \int_\Gamma \int_\mathbb{R}f(\gamma)\,d\langle e_\gamma x(s),x(s)\rangle\,ds.
\end{eqnarray*}
We in particular also have that $f(a)(s)=f(a(s))$ for almost every $s\in \Gamma$.
\end{proposition}

\subsubsection{The Borel functional calculus for positive operators affiliated to $L^\infty(\Gamma)\overline{\otimes}\M$}  It remains to describe the set of self-adjoint 
operators affiliated to $\N$ as well as give some insight into the Borel functional calculus for these operators. For our purposes it will be enough to do so for positive 
densely defined operators affiliated to $\N$. It is well known that a densely defined positive operator $a$ is affiliated to $\N$ if and only if all of the spectral 
projections of $a$ belong to $L^\infty(\Gamma)\overline{\otimes}\M$. Given a densely defined affiliated positive operator $a$ and positive real numbers  
$\gamma_0>\gamma_1\geq 0$, the operator $a\chi_{[0,\gamma_0]}(a)$ is an element of $L^\infty(\Gamma)\overline{\otimes}\M$ for which we have that 
$a\chi_{[\gamma_1,\gamma_0]}(a)= a\chi_{[\gamma_0,0]}(a)\chi_{[\gamma_1,\gamma_0]}(a\chi_{[\gamma_0,0]}(a))$. 

This quickly leads to the conclusion that the sequence of operators $a_n = a\chi_{[0,n]}(a)\in L^\infty(\Gamma)\overline{\otimes}\M$ increase to $a$. For each $n$ and each 
$s\in \Gamma$ we have that $a_n(s) =(a\chi_{[n,0]}(a))(s)$. So on the other hand the sequence of operators $(a_n(s))$ will for each $s$ increase to an operator $a(s)$ 
affiliated to $L^\infty(\Gamma)\overline{\otimes}\M$. The operator $a$ may therefore be regarded as an operator-valued function $s\to a(s)$ on $\Gamma$. It now behoves us to 
clarify the properties of this function. Since for any $x\in \mathrm{dom}(a)\subseteq L^2(\Gamma;H)$ we will then have that $$\langle ax,x\rangle = \sup_n\langle a_nx,x\rangle 
= \sup_n\int_\Gamma \langle a(s)x(s),x(s)\rangle\,ds =\int_\Gamma \langle a(s)x(s),x(s)\rangle\,ds<\infty$$it clearly follows that almost every $a(s)$ must then also be densely 
defined.

As far as the spectral projections of these operators are concerned we note that for any $s\in \Gamma$ and any pair of positive real numbers $\gamma_0>\gamma_1\geq 0$ the 
aforementioned observations and definitions ensure that we will for any $n>\gamma_0$ have that $\chi_{[\gamma_1,\gamma_0]}(a)(s) = \chi_{[\gamma_1,\gamma_0]}(a_n)(s) = 
\chi_{[\gamma_1,\gamma_0]}(a_n(s)) = \chi_{[\gamma_1,\gamma_0]}(a(s))$ and hence that $\chi_{[\gamma_1,\gamma_0]}(a)(s) = \chi_{[\gamma_1,\gamma_0]}(a(s))$. On letting 
$\gamma_0\nearrow\infty$, it is clear that we also have that $\chi_{[\gamma_1,\infty)}(a)(s) = \chi_{[\gamma_1,\infty)}(a(s))$. It now follows from these facts (and an 
argument similar to the one preceding \cite[Theorem 5.2.8]{KR}) that for $a(s)$ as defined above we in fact have that $\chi_B(a)(s)=\chi_B(a(s))$ for any Borel subset 
$B\subset[0,\infty)$. We therefore obtain 

\begin{proposition} Any positive self-adjoint operator $a$ affiliated to $L^\infty(\Gamma)\overline{\otimes}\M$ may be written as a function on $\Gamma$ taking values in the 
set of self-adjoint positive operators affiliated to $\M$ in such a manner that for any Borel subset $B\subset[0,\infty)$ and any $s\in\Gamma$ we have that 
$\chi_B(a)(s)=\chi_B(a(s))$. 
\end{proposition}

Having noted the above, we may now use \cite[Lemma 5.6.31]{KR} to expose the essentials of the Borel functional calculus for positive densely defined operators affiliated to 
$\N$. For the sake of simplicity of notation we shall depart from the notation in \cite{KR} and simply write $ae$ for the strong product of such an affiliated operator $a$ 
and one of its spectral projections. Let us firstly assume that $f$ is a bounded complex valued Borel function. Using \cite[Lemma 5.6.31]{KR} and what we already know to be 
the case for positive elements of $\N$, it follows that $$f(a)(s)\chi_{[0,n]}(a)(s) = f(a_n)(s)\chi_{[0,n]}(a_n)(s) = f(a_n(s))\chi_{[0,n]}(a_n(s)) = 
f(a(s))\chi_{[0,n]}(a(s)).$$ Letting $n\nearrow \infty$ now shows that $f(a)\chi_{[0,n]}(a)$ and $f(a(s))\chi_{[0,n]}(a(s))$ ($s\in\Gamma$) respectively increase to 
affiliated operators $f(a)$ and $f(a(s))$ for which we have that $f(a)(s)=f(a(s))$ for almost every $s\in\Gamma$. If now we apply Proposition \ref{Borel-bd} to the operators 
$a_n$ and then invoke the monotone convergence theorem, it follows that we will for any $x\in \mathrm{dom}(a)\subset L^2(\Gamma;H)$ have that 
\begin{eqnarray*}
\int_\Gamma \langle f(a(s))x(s),x(s)\rangle\,ds &=& \langle f(a)x,x\rangle\\
&=& \int_{[0,\infty)}f(\gamma)\,d\langle e_\gamma x,x\rangle\\ 
&=& \int_\Gamma \int_{[0,\infty)}f(\gamma)\,d\langle e_\gamma(s) x(s),x(s)\rangle\,ds.
\end{eqnarray*}

Given a general non-negative almost everywhere finite Borel function $f$ on $\sigma(a)$, the corresponding statement for such functions may then be obtained by writing $f$ as 
the monotone limit of an increasing sequence $(f_n)$ of non-negative Borel functions and once again invoking the monotone concergence theorem. We therefore obtain the 
following

\begin{proposition}\label{Borel-unbd} Let $a$ be a positive self-adjoint operator affiliated to $L^\infty(\Gamma)\overline{\otimes}\M$. For any non-negative Borel function $f$
on $[0,\infty)$ which is bounded on compact sets there is a unique selfadjoint positive operator $f(x)$ with domain
$$\mathrm{dom}(f(a))=\left\{x\in L^2(\Gamma;H)\colon   \int_{0}^{\infty}|f(\gamma)|^2 d\langle e_\gamma x, x\rangle<\infty\right\}$$given by 
$$\langle f(a)x, x \rangle = \int_\Gamma \langle f(a(s))x(s), x(s) \rangle\,ds = \int_\Gamma\int_{[0,\infty)}f(\gamma) d\langle e_\gamma(s)x(s), x(s)\rangle\,ds = 
\int_{[0,\infty)}f(\gamma) d\langle e_\gamma\xi, \xi \rangle.$$We in particular also have that $f(a)(s)=f(a(s))$ for almost every $s\in \Gamma$.
\end{proposition}

\subsubsection{The trace on $L^\infty(\Gamma)\overline{\otimes}\M$}
The predual of $L^\infty(\Gamma)\overline{\otimes}\M$ similarly corresponds to the space $L^1(\Gamma;\M_*)\equiv L^1(\Gamma)\otimes_\pi \M_*$ of measurable $\M_*$ valued 
functions ${\omega}$ on $\Gamma$ normed by the norm $\int_\Gamma\|{\omega}(s)\|\,d\nu(s)<\infty$. The dual action is described in \cite[Theorem IV.7.17]{Tak}. We may of 
course without loss of generality rather realise the predual as $L^1(\Gamma)\otimes_\pi L^1(\M,\tau)\equiv L^1(\Gamma;L^1(\M,\tau))$, in which case the dual action will for 
say ${a}\in L^\infty(\Gamma)\overline{\otimes}\M$ and ${x}\in L^1(\Gamma;L^1(\M,\tau))$ then be given 
by $$\langle {a},{x}\rangle = \left(\int_\Gamma\otimes\tau\right)({a}{x})=\int_\Gamma\tau({a}(s){x}(s))\,d\nu(s).$$Given 
${a}\in L^\infty(\Gamma)\overline{\otimes}\M$, we may select a net of projections $(p_\gamma)\subset L^1(\M,\tau)\cap L^\infty(\M,\tau)$ increasing to $\I$ and a net 
$(K_\beta)$ of compact neighbourhoods increasing to $\Gamma$ (in which case $(\chi_{K_\beta})\subset L^1(\Gamma)\cap L^\infty(\Gamma)$ similarly increases to the constant 
function 1). The simple tensors $\chi_{K_\beta}\otimes p_\gamma$, and hence also $(\chi_{K_\beta}\otimes p_\gamma){a}$, will then all belong to the predual 
$L^1(\Gamma;L^1(\M,\tau))$. The dual action formula given above then ensures that $$\left(\int_\Gamma\otimes\tau\right)({a}^*(p_\gamma\otimes \chi_{K_\beta}){a}= 
\int_{K_\beta}\tau({a}(s)^*p_\gamma{a}(s))\,d\nu(s).$$By first letting $p_\gamma\nearrow \I$ and then $K_\beta\nearrow \Gamma$, we obtain that 
$$\left(\int_\Gamma\otimes\tau\right)({a}^*{a})= \int_\Gamma\tau({a}(s)^*{a}(s))\,d\nu(s)\mbox{ for all }{a}\in 
L^\infty(\Gamma)\overline{\otimes}\M.$$

For the above trace the associated $\tau$-measurable operators are of course those densely define affiliated operators $a$ affiliated to  
$L^\infty(\Gamma)\overline{\otimes}\M$ for which there exists some $\lambda>0$ such that $(\int_\Gamma\circ\tau)(\chi_{(\infty,\lambda)}(|a|)) =  
\int_\Gamma\tau(\chi_{(\infty,\lambda)}(|a|(s)))\,ds <\infty$. We then clearly also have that $\tau(\chi_{(\infty,\lambda)}(|a|(s)))<\infty$ almost everywhere, showing that 
for almost every $s$, each $a(s)$ will also be $\tau$-measurable.

\subsection{A review of group actions}\label{vNA-gac}

Let $\M$ be a semifinite von Neumann algebra equipped with a faithful normal semifinite trace $\tau$ and $G$ some mutiplicative locally compact group which acts on $\M$ by 
means of a group of *-automorphisms $\alpha_s$ ($s\in G$) which respect the group action in that $\alpha_s\alpha_t=\alpha_{st}$ and $\alpha_e=\mathrm{Id}_{\M}$. Guided by the 
classical context, one would also require that $\tau_\M\circ\alpha_g =\tau_\M$ for each $g\in G$ and also insist that for each $a\in \M$,  $G\to \M: g\to \alpha_g(a)$ is a 
continuous map when $\M$ is equipped with its $w*$-topology. It easily follows from \cite[Proposition 4.7]{L} and the fact that $\tau_\M\circ\alpha_g =\tau_\M$ for each 
$g\in G$, ensures that each $\alpha_g$ canonically and continuously extends to a *-automorphism on $\widetilde{\M}$. We will retain the notation $\alpha_g$ for this 
extension. Given any $a\in \M$, $b\in L^1(\M)$, we then further have that $$\tau(\alpha_g(a)b)=\tau\circ\alpha_{g^{-1}}(\alpha_g(a)b)=\tau(a\alpha_{g^{-1}}(b)).$$Hence the 
action of each $\alpha_g$ on $\M$ appears as the dual of the action of $\alpha_{g^{-1}}$ on $L^1(\M)\equiv \M_*$. From the discussion on page 281 of \cite{Bla} it is now 
clear that on $L^1(\M)$ the map $g\to \alpha_g(a)$ is point-norm continuous for each $a\in L^1(\M)$, whereas in its action on $\M$, the same map is point-($\sigma$-strong) 
continuous. We summarise the above discussion in the following lemmata. 

\begin{lemma}
Let $G\to \mathrm{Aut}(\M):g\to\alpha_g$ be a representation of the locally compact group $G$ as a group of *-automorphisms on $\M$satisfying $\tau_\M\circ\alpha_g =\tau_\M$ 
for each $g\in G$. Also let $G\to \M: g\to \alpha_g(a)$ be point-w* continuous for each $a\in \M$. Then each $\alpha_g$ canonically extends to a *-automorphism on 
$\widetilde{\M}$ which is homeomorphic with respect to the topology of convergence in measure, and which maps $L^1(\M)$ into $L^1(\M)$. Moreover the action of each $\alpha_g$ 
on $\M$ appears as the dual of the action of $\alpha_{g^{-1}}$ on $L^1(\M)\equiv \M_*$. In addition for any fixed $a\in L^1(\M)$, the map $G\to L^1(\M):g\to\alpha_g(a)$ is 
point-norm continuous.
\end{lemma}

\begin{lemma}
The von Neumann algebra $\M$ is complete in each of the $\sigma$-strong and $\sigma$-strong* topologies, and fully complete in the Arens-Mackey topology. Moreover for each 
$a\in M$, the map $G\to M: g\to \alpha_g(a)$ will be continuous with respect to each of these topologies on $M$.
\end{lemma} 

\begin{proof}
The first claim is contained in \cite[I.3.2.2]{Bla}. Next note that any dual of a Frech\'et space is fully complete when equipped with the Arens-Mackey topology (see p. 123 
of \cite{RR}).  

Next let $a\in \M$ be given. The point-($\sigma$-strong) continuity of the map $G\to \M: g\to \alpha_g(a)$ is stated on page 281 of \cite{Bla}. To see that this map is point-(
$\sigma$-strong*) continuous, simply apply \cite[Lemma II.2.5 \& Prop II.4.1]{Tak} to the maps produced by the real and imaginary parts of $a$. For any $a\in \M$, the set 
$\{\alpha_g(a): g\in \M\}$ is of course norm-bounded. Hence the fact that $G\to \M: g\to \alpha_g(a)$ is point-(Arens-Mackey) continuous, now follows from 
\cite[Theorem II.5.7]{Tak}.
\end{proof}

\subsection{Calderon's transfer principle}\label{CTP} 

For any Radon measure space $(\Gamma,\nu)$ the space $L^{loc}(\Gamma)$ of \emph{locally integrable functions} is defined to be the space of measurable functions which are 
integrable on measurable sets of finite measure. With this space as starting point, the following concept of transferable operators was introduced in \cite{dBL}:

\begin{definition}[Transferable operators-$\beta$] Let $T$ be a continuous linear operator having the class of locally integrable functions over $G$ as domain and $C(G)$, the 
space of continuous functions on $G$ given the compact-open topology, as range. It is called \emph{transferable} if it satisfies the following conditions:
\begin{enumerate}
\item $T$ is \emph{semilocal}: there exists an open neighbourhood $U$ of $e \in G$ with compact closure such that if $\mathrm{supp}(f)$ is contained in a set $V$ , then 
$\mathrm{supp}(T(f))$ is contained in $U^{-1}V$.
\item $T$ is \emph{translation invariant}: for all $g \in G$ and $f \in L^{loc}(G)$, $\lambda_g\circ T(f) = T \lambda_g(f)$, where $\lambda_g$ is the left translation 
operator $\lambda_g(f)(s) = f(g^{-1}s)$.
\end{enumerate}
\end{definition} 

Given a $\sigma$-finite measure space $(\Omega,\omega)$ which admits an integrable group action, the first step in Calderon's transfer principle is to use the information 
encoded in the action of a transferable operator $T$, to construct a related operator $T^\#$ in the measure space context. This process proceeds as follows: For any 
$f\in L^{1+\infty}(\Omega)$ and any $h\in L^{loc}(\Omega)$, define $h\otimes_\alpha f$ by $(\omega, t)\to f(\alpha_t(\omega))h(t)$ for all 
$\omega\in \Omega$ and $t\in G$.  
\begin{itemize}
\item The first step is note that the prescription $\iota_\alpha:f\to 1\otimes_\alpha f$ yields a map from $L^{loc}(\Omega)$ to $L^{r-loc}(G\times\Omega)$ 
(the space of all measurable functions on $(G\times \Omega)$ that are integrable on measurable rectangles of finite measure).
\item The space $L^{r-loc}(G\times \Omega)$ can then be shown to be a copy of the Bochner type space $L^{loc}(\Omega, L^{loc}(G))$ with the operator $T$ inducing an 
operator $\widetilde{T}$ from $L^{loc}(\Omega, L^{loc}(G))$ to $L^{loc}(\Omega, C(G))$ by means of composition. 
\item With $\delta_e$ denoting point evaluation at the group unit $e$, $\delta_e\otimes \I$ will now map 
$C(G)\widehat{\otimes}_\pi L^{loc}(\Omega)\equiv L^{loc}(\Omega, C(G))$ to $L^{loc}(\Omega)$.
\item The composition $(\delta_e\otimes\I) \circ \widetilde{T} \circ \iota_{\alpha}$ then yields a linear map $T^\#$ from $L^{1+\infty}(\Omega)$ into $L^{loc}(\Omega)$.
\end{itemize}

With $T^\#$ computed the next task is compute the weak/strong type of $T^\#$ from that of $T$. Once that is a known a version of the Banach principle can be used to extract 
ergodic convergence type theorems for sequences $(T_n^\#)$ of such transferred operators.

In the noncommutative world there is however as yet no satisfying non-commutative analogue of the space $L^{loc}$. Our starting point will therefore be the space 
$L^{1+\infty}(G)$, with the objective being to construct a related operator on $L^{1+\infty}(\M,\tau)$. 
\begin{quote}\emph{So for us the domain of a transferable operator will not be $L^{loc}(G)$, but $L^{1+\infty}(G)$.}\end{quote}
By analogy with the classical setting we formally define a map $\otimes_\alpha$ on $L^{1+\infty}(G)\otimes L^{1+\infty}(\M,\tau)$ by the prescription 
$(f\otimes_\alpha a)(g)= f(g)\alpha_g(a)$. For this to make sense we need to see when terms of the form $(f\otimes_\alpha a)$ even belong to $L^{1+\infty}(\M,\tau)$. A useful 
fact that that will aid us in this regard is the fact that for any semifinite von Neumann algebra $\M$ equipped with an fns trace $\tau$ we have that 
$L^{1+\infty}(\M,\tau)= L^\infty(\M,\tau) + L^1(\M,\tau)$ (see \cite[Proposition 2.5]{DDdP-Kothe}). Thus if we are given $f\in L^{1\cap\infty}(G)$ and 
$a \in L^{1+\infty}(\M,\tau)$, then on writing $a$ as $a=a_\infty+a_1$ where $a_\infty\in L^{\infty}(\M,\tau)$ and $a_1\in L^{1}(\M,\tau)$, it then follows from the 
discussion in subsection \ref{vNA_MG} that $f\otimes a_\infty \in L^\infty(G)\overline{\otimes}\M$ with $f\otimes a_1\in L^1(G)\otimes_\pi L^1(\M,\tau)\subset 
L^1(L^\infty(G)\overline{\otimes}\M)$. That in turn ensures that $f\otimes a =(f\otimes a_\infty)+(f\otimes a_1)\in L^{1+\infty}(L^\infty(G)\overline{\otimes}\M)$. The 
following Lemma now provides the information we need to conclude this analysis of the bilinear map $\otimes_\alpha$. 

\begin{lemma}\label{2.4} For every $f\in L^{1\cap\infty}(G)$ and $a\in L^{1+\infty}(\M)$ the rearrangements of $f\otimes a$ and $f\otimes_\alpha a$ are equimeasurable 
elements of $L^{1+\infty}(L^\infty(G)\overline{\oplus}\M)$.
\end{lemma}

\begin{proof} Given $\lambda>0$ we leave it as an exercise to verify that
$$\chi_{(\lambda,\infty)}(|f\otimes a|)(s)=\chi_{(\lambda,\infty)}(|f(s)a|)\mbox{ and }\chi_{(\lambda,\infty)}(|f\otimes_\alpha a|)(s)= 
\chi_{(\lambda,\infty)}(|f(s)\alpha_s(a)|).$$On passing to the Borel functional calculus, it is clear that we may assume that $f(t)\neq 0$ on the spectral subspaces 
corresponding to $\chi_{(\lambda,\infty)}(|f(s)a|)$ and $\chi_{(\lambda,\infty)}(|f(s)a|)$ respectively. Therefore 
$$\chi_{(\lambda,\infty)}(|f(s)a|)= \chi_{((\lambda/|f(s)|),\infty)}(|a|)\mbox{ and }\chi_{(\lambda,\infty)}(|f(s)\alpha_s(a)|) = 
\chi_{((\lambda/|f(s)|),\infty)}(|\alpha_s(a)|).$$Now observe that the fact $\tau\circ\alpha_s=\tau$, ensures that $a$ and $\alpha_s(a)$ are equimeasurable, which then 
ensures that $$\tau(\chi_{(\lambda,\infty)}(|f(s)a|)) =\tau(\chi_{((\lambda/|f(s)|),\infty)}(|a|))= \tau(\chi_{((\lambda/|f(s)|),\infty)}(|\alpha_s(a)|) = 
\tau(\chi_{((\lambda),\infty)}(|f(s)\alpha_s(a)|)$$ which in turn leads to the conclusion that 
\begin{eqnarray*}
\left(\int_G\circ\tau\right)(\chi_{(\lambda,\infty)}(|f\otimes a|)) &=& \int_G\tau(\chi_{(\lambda,\infty)}(|f(s)a|))\,ds\\
&=& \int_G\tau(\chi_{(\lambda,\infty)}(|f(s)\alpha_s(a)|))\,ds\\
&=& \left(\int_G\circ\tau\right)(\chi_{(\lambda,\infty)}(|f\otimes_\alpha a|)).
\end{eqnarray*}
This then proves the claim.
\end{proof}

Given $f\in L^{1\cap\infty}(G)$ and $a\in L^{1+\infty}(\M)$, it now immediately follows from the basic axioms of rearrangement invariant Banach function spaces that both 
$f\otimes a$ and $f\otimes_\alpha a$ belong to $L^{1+\infty}(L^\infty(G)\overline{\otimes}\M)$ and that in fact $\|f\otimes a\|_{1+\infty}=\|f\otimes_\alpha a\|_{1+\infty}$.

Given any compact neighbourhood $K$ of $e\in G$, the embedding $j_K:L^{1+\infty}(\M,\tau)\to L^{1+\infty}(L^\infty(G)\overline{\otimes}\M):a\to \chi_K\otimes_\alpha a$, where 
$(\chi_K\otimes_\alpha a)(g)= \chi_K(g)\alpha_g(a)$ for each $g\in G$, is therefore well-defined. The next step would then be to show that the map $T\otimes \mathrm{Id}$ 
extends to a map $\widetilde{T}$ on $L^{1+\infty}(L^\infty(G)\overline{\otimes}\M)$ which then gets composed with the above embedding. As before a version of point-evaluation 
will then yield a ``localised'' transferred operator $T^\#_K$. Formally the operator $T^\#_K$ is then of the form 
$$a\to \chi_K\otimes_\alpha a\to \widetilde{T}(\chi_K\otimes_\alpha) \to (\delta_e\otimes \mathrm{Id})\widetilde{T}(\chi_K\otimes_\alpha).$$If in conclusion we can show that 
as $K$ increases to $G$ the operators $T^\#_K$ actually converge to a unique operator, that operator then qualifies to be called the transferred operator $T^\#$ of $T$. Once 
the existence of the transferred operator has been established, the next task would then be to compute the weak/strong type of $T^\#$ from that of $T$. 

In the most important classical cases, transferable operators are of the form $T_k:f\to k*f$ where $k$ is a positive integrable bounded function of the form 
$k=\frac{1}{|U|}\chi_U$ and $U\subset G$ a Borel set of finite measure. Unless otherwise stated we shall therefore restrict ourselves to such operators where we will 
typically assume that $k\in L^{1\cap\infty}(G)_+$ with $\|k\|_\infty\leq 1$ and $\|k\|_1=1$. As we shall shortly see such an operator $T$ will induce a positive identity 
preserving operator on $L^\infty(G)$ for which we have that $\int_G T(f)\,dt\leq \int_Gf\,dt$. Since we are here in a ``commutative'' context, $T$ is known to then in fact be 
completely positive. It then follows from Theorem 5.1 of \cite{HJX} that the map $\widetilde{T}=T\otimes\mathrm{Id}$ is an identity preserving positive operator on 
$L^\infty(G)\overline{\otimes}\M$ for which we have that $(\int_G\circ\tau)\circ (T\otimes \mathrm{Id}) \leq (\int_G\circ\tau)$. Since here $L^\infty(G)\overline{\otimes}\M$ 
is in fact semifinite (with fns trace $(\int_G\otimes \tau)$), we may then use real interpolation \cite[\S 7.2]{GLbook} to extend the claims of \cite[Theorem 5.1]{HJX} to 
the claim that for any monotone Riesz-Fischer norm $\bm{\rho}$ we have that $\widetilde{T}$ extends to a contractive map on $L^{\bm{\rho}}(\M\overline{\otimes}L^\infty(G))$. 
Note that Theorem 5.1 of \cite{HJX} also ensures that $\widetilde{T}$ is in fact normal on $\M\overline{\otimes}L^\infty(G)$! For the specific case where $L^{\bm{\rho}} 
=L^X=L^Y$ for some Young functions $X$ and $Y$, the weak/strong type of $T^\#$ can then be extracted as follows:
\begin{itemize}
\item For Young functions $A$ and $B$, we firstly need to make sense of spaces of the form $L^B(G;L^A(\M))$ and find criteria under which such spaces continuously embed into 
$L^X(L^\infty(G)\overline{\otimes}\M)$. For any fixed compact neighbourhood $K$ of $e\in G$, the prescription $L^A(\M)\to L^B(G;L^A(\M)):f\mapsto \chi_K\otimes f$ can then be 
seen to yield a continuous embedding of $L^A(\M)$ into $L^B(G;L^A(\M))$ and hence also into $L^X(L^\infty(G)\overline{\otimes}\M)$.
\item We next note that the equi-measurability of twisted tensor $\chi_K\otimes_\alpha f$ to $\chi_K\otimes f$, ensures that the map $\chi_K\otimes f \mapsto 
\chi_K\otimes_\alpha f$ is a linear isometry on the subspace of $L^X(\M\overline{\otimes}L^\infty(G))$ generated by these simple tensors. 
\item The final step is to then apply the map $\widetilde{T}$ to this subspace of $L^X(L^\infty(G)\overline{\otimes}\M)$ and to use a noncommutative version of 
\cite[Theorem 4.3]{L-QM} to extract the weak/strong type of $T^\#$.
\end{itemize}

\subsection{Emerging challenges} 
The challenges we face are then as follows: We firstly need to prove the actual existence of the map $T^\#$. As far as the protocol for computing the weak/strong type of 
$T^\#$ is concerned we need to start with checking that maps of the type $T_k:f\to k*f$ where $k\in L^{1\cap\infty}(G)_+$ with $\|k\|_\infty\leq 1$ and $\|k\|_1=1$ are indeed 
completely positive. This we do below. We then next need to make sense of spaces of the type $L^B(G;L^A(\M))$ and investigate conditions under which they embed in spaces of 
the form $L^X(L^\infty(G)\overline{\otimes}\M)$. We finally need to develop the technology that will allow us to prove a noncommutative version of \cite[Theorem 4.3]{L-QM}. 
We start this journey of discovery by verifying the claims made about transferable operators induced by convolution.

\begin{proposition}\label{T-posop} Let $k$ be a non-negative integrable bounded function for which we have that $\|k\|_\infty, \|k\|_1\leq 1$. Then $T_k:f\to k*f$ defines a 
positive operator on on $L^\infty(G)$ for which we have that $T(1)\leq 1$ and $\int_G\circ T\leq \int_G$ on $L^\infty_+$.
\end{proposition}

\begin{proof} For any $f\in L^\infty_+$ we clearly have that $(k*f)(t)=\int_G k(s)f(s^{-1}t)\,dy$ is a non-negative function. If $f=1$, then $T(1)(t)=(k*1)(t)= 
\int_G k(s)\,ds\leq 1$. Regarding the final claim we first observe that for any $f\in L^\infty_+$ with $\int_Gf\,dt<\infty$ we have that $\int_GT(f)\,dt = 
\int_Gk*f\,dt\leq \infty =\|k*f\|_1\leq \|k\|_1\|f\|_1\leq \int_Gf\,dt$. Given an arbitrary $f\in L^\infty_+$, select a net 
$(f_\alpha)\subset(L^\infty\cap L^1)_+$ increasing to $f$ (in the von Neumann algebraic sense). We know that then $\int_GT(f_\alpha)\,dt \leq \int_Gf_\alpha\,dt \leq 
\int_Gf\,dt$ for each $\alpha$. So if we can show that $\sup_\alpha\int_GT(f_\alpha)\,dt =\int_GT(f)\,dt$, we are done. Let $x\in L^1(G)_+$ be given. Here $x$ corresponds to 
a positive normal functional on $L^\infty$. We will for any $h\in L^\infty$ then have that 
$$\langle T(h), x\rangle =\int_G (k*h)(s)x(s)\,ds = \int_G x(s)\int_Gk(t)h(t^{-1}s)\,dt\,ds = \int_G k(t)\int_Gx(s)h(t^{-1}s)\,ds\,dt.$$Now set $s=tu$ (equivalently 
$t=us^{-1}$ and use the left-invariance of Haar measure to see that we will on first setting $u=t^{-1}s$ and then $v=t^{-1}$, have that 
$$\int_G k(t)\int_Gx(s)h(t^{-1}s)\,ds\,dt= \int_G k(t)\int_Gx(tu)h(u)\,du\,dt =\int_G h(u)\int_Gk(t)x(tu)\,dt\,du =$$ 
$$\int_G h(u)\int_Gk(t)x(tu)\Delta(t)\,d(t^{-1})\,du = \int_G h(u)\int_Gk(v^{-1})x(v^{-1}u)\Delta(v^{-1})\,dv\,du.$$On defining $k^*$ as in \S 8.6.2 of \cite{GLbook} we 
therefore have that 
$$\langle T(h), x\rangle =\int_G k(t)\int_Gx(s)h(t^{-1}s)\,ds\,dt =\int_G h(u)\int_Gk^*(v)x(v^{-1}u)\Delta(v^{-1})\,dv\,du = \int_G h(u)(k^**x)(u)\,du.$$We leave it as an 
exercise to verify that $(f_\alpha(k^**x))$ increases to $f(k^**x)$ from which it then follows that 
$$\sup_\alpha\langle T(f_\alpha), x\rangle =\sup_\alpha\int_G f_\alpha(u)(k^**x)(u)\,du =\int_G f(u)(k^**x)(u)\,du = \langle T(f), x\rangle.$$Since this equality holds for 
all positive normal functionals on $L^\infty$, we have that $\sup_\alpha\langle T(f_\alpha) = T(f)$ which proves the claim.
\end{proof}

\subsection{Strategies for deriving pointwise ergodic convergence theorems} 
We turn to the derivation of pointwise ergodic theorems. Given a sequence $(T^\#_n)$ of transferred operators, our first task will be to define some sort of maximal operator 
$\mathcal{T}^\#_{max}$ induced by this sequence. The objective is to show that if this maximal operator is say weak type $(X,Y)$, the transferred operators $T^{\#}_n$ will  
then satisfy an ergodic convergence theorem with $(T^{\#}_nf(\omega))\subset M_Y(\M)$ converging in the $\sigma(M_Y(\M),\Lambda_{Y^*}(\M))$-topology for every $f\in L^X(\M)$. 
The strategy for achieving this is as follows:
\begin{enumerate}\item We shall under acceptable restrictions firstly identify a dense subset $D$ in the domain of $\mathcal{T}^\#_{max}$ for which the convergence of 
$(T^{\#}_nf)$ can be verified for all $f\in D$.
\item We shall then next use an appropriate Banach type Principle to extend the convergence of step (1) to the full domain of $T^{\#}$.
\end{enumerate}

\part{Existence of the noncommutative transfer operator}

\section{Constructing the quantum transfer operator}\label{deftran}

For the sake of access to real interpolation, the transferred operator $T^{\#}$ should ideally have $L^{1+\infty}(\M)$ as its domain. As noted in subsection \ref{CTP}, this 
requirement forces us to construct the transferred operator on $L^{1+\infty}(\M,\tau)$ in a somewhat indirect way. We start our analysis by first collecting some known facts 
regarding the projective tensor product of complete locally convex spaces, after which we introduce a new locally convex topology on $L^{1+\infty}(\M)$. 

\begin{proposition}
Let $E_0$ and $E_1$ be complete locally convex spaces whose topologies are respectively produced by the classes of seminorms $\{p_\alpha: \alpha\in \Lambda\}$ and 
$\{q_\beta: \beta\in \Gamma\}$. Next let $X_\alpha$ and $Y_\beta$ respectively be the completions of the spaces $(E_0/[p_\alpha^{-1}(0)], \widehat{p}_\alpha)$ and 
$(E_1/[q_\beta^{-1}(0)], \widehat{q}_\beta)$ where $\widehat{p}_\alpha$ and $\widehat{q}_\beta$ are the norms canonically induced by $p_\alpha$ and $q_\beta$. We write 
$E_0\widetilde{\otimes}_\pi E_1$ for the completion of algebraic tensor product $E_0{\otimes} E_1$ equipped with the projective topology. The topology on 
$E_0\widetilde{\otimes}_\pi E_1$ is produced by the seminorms $\pi_{\alpha, \beta} = p_\alpha\otimes q_\beta$, and under this topology, $E_0\widetilde{\otimes}_\pi E_1$ may 
also be realised as the projective limit of the tensor products $X_\alpha\widetilde{\otimes}_\pi Y_\beta$. In the case that $E_0=L^1(\Omega, \Sigma,\nu)$ where 
$(\Omega, \Sigma, \nu)$ is a $\sigma$-finite measure space, the seminorms producing the topology of $E_0\widetilde{\otimes}_\pi E_1$ are just $\pi_{\beta}(\cdot) = 
\int_\Omega q_\beta(\cdot)\, d\nu$ ($\beta\in\Gamma$). 
\end{proposition}

\begin{proof}
By \cite[2.9.1 \& 3.4.7]{Jar} $E_0$ and $E_1$ are respectively reduced projective limits of the spaces $X_\alpha$ and $Y_\beta$. Hence by \cite[15.4.2]{Jar}, the space 
$E_0\widetilde{\otimes}_\pi E_1$ is a reduced projective limit of the spaces $(E_0/[p_\alpha^{-1}(0)], \widehat{p}_\alpha) 
\widetilde{\otimes}_\pi (E_1/[q_\beta^{-1}(0)],\widehat{q}_\beta)$. Apart from the final claim regarding the case $E_0=L^1(\Omega, \Sigma,\nu)$, the balance of the claims 
follow from the discussion in \cite[\S 15.4]{Jar}. The final claim regarding the seminorms in case $E_0=L^1(\Omega, \Sigma,\nu)$, follows from the discussion on page 339 of 
\cite{Jar}.  
\end{proof}

Given a projection $e\in \M$ with $\tau(e)<\infty$, we will for any $a\in \M$ then have that $|eae|^2\leq \|eae\|^2e\leq \|a\|^2e$. Thus clearly 
$\tau(|eae|)\leq \|a\|\tau(e)<\infty$. In other words $eae\in L^1(\M)$ for any $a\in \M$. In view of this fact we now introduce the topology of (bilateral) local 
integrability on $L^{1+\infty}(\M)$.

\begin{definition} 
With each projection $e\in \M$ with $\tau(e)<\infty$, we associate the seminorm $p_e(a)=\|eae\|_1$ ($a\in L^{1+\infty}(\M)$). It is an exercise to see that the topology 
produced by these seminorms is Hausdorff. We shall call the locally convex topology produced on $L^{1+\infty}(\M)$ by these seminorms, the topology of (bilateral) local 
integrability on $L^{1+\infty}(\M)$, and will where convenient write (\emph{loc}) for this topology. In essence convergence in this topology corresponds to convergence in 
$L^1(e\M e)$ for each projection $e$ with $\tau(e)<\infty$.
\end{definition}

\begin{theorem}
Let $\bm{\rho}$ be a Banach function norm on $M^0(\mathbb{R}_+)$ with the Fatou property. Then the closed unit ball of $L^{\bm{\rho}}(\M,\tau)$ is complete under the topology 
of (bilateral) local integrability.
\end{theorem}

\begin{proof}
Let $\{f_\lambda\}$ be a Cauchy net in the closed unit ball of $L^{\bm{\rho}}(\M,\tau)$. Then for each projection $e\in \M$ with 
$\tau(e)<\infty$, we have that $\{ef_\lambda e\}$ is a Cauchy net in $L^1(e\M e)$, and hence converges to some element $g_e$ of $L^1(e\M e)$. But then by known results 
$\{ef_\lambda e\}$ will converge to $g_e$ in the topology of convergence in measure on $\widetilde{e\M e}$. However convergence in measure for each $\widetilde{e\M e}$ where 
$\tau(e)<\infty$ corresponds to the topology of local convergence in measure as defined in \cite{DDST}. But by \cite[Corollary 4.5]{DDST} the unit ball of 
$L^{\bm{\rho}}(\M,\tau)$ is complete in this topology. Hence there exists $f_0$ in the unit ball of $L^{\bm{\rho}}(\M,\tau)$ so that $\{ef_\lambda e\}$ converges in measure 
to $ef_0e$ for each $e$ with $\tau(e)<\infty$. This ensures that $ef_0e = g_e$ for each $e$ with $\tau(e)<\infty$, and hence that $\{f_\lambda\}$ converges to $f_0$ in the 
topology of (bilateral) local integrability.
\end{proof}

One more critical ingedient we need is a clear understanding of the continuity properties of the given group action with respect to the topology we have introduced.

\begin{proposition}
For each $a\in L^{1+\infty}(\M)$, the map $G\to L^{1+\infty}(\M):g\to\alpha_g(a)$ is point-(\emph{loc}) continuous.
\end{proposition}

\begin{proof}
Since algebraically $L^{1+\infty}(\M)=L^1(\M)+\M$, we may separately consider the cases $a\in L^1(\M)$ and $a\in \M$. For the first of these the claim is a trivial 
consequence of the observations made in the introduction. Let $a\in \M$ be given with $e$ a projection with finite trace. It is enough to show that $\tau(e|\alpha_g(a)-a|e) 
\to 0$ as $g\to 1$. Firstly note that 
$$(e|\alpha_g(a)-a|e)^2= e|\alpha_g(a)-a|\,e\,|\alpha_g(a)-a|e \leq e|\alpha_g(a)-a|^2e=|(\alpha_g(a)-a)e|^2.$$Hence $\tau(e|\alpha_g(a)-a|e)\leq \tau(|(\alpha_g(a)-a)e|)$. 
If therefore we can show that $\tau(|(\alpha_g(a)-a)e|)\to 0$ as $g\to 1$, we are done. To see that this is the case, observe that by H\"older's inequality,
\begin{eqnarray*}
\tau(|(\alpha_g(a)-a)e|) &\leq& \tau(|(\alpha_g(a)(e-\alpha(e))|)+\tau(|\alpha_g(ae)-ae|)\\
&\leq& \|\alpha_g(a)\|_\infty.\tau(|e-\alpha(e)|)+\tau(|\alpha_g(ae)-ae|)\\
&=& \|a\|_\infty.\tau(|e-\alpha(e)|)+\tau(|\alpha_g(ae)-ae|)
\end{eqnarray*}
By the first part of the proof $\tau(|e-\alpha(e)|)\to )$ and $\tau(|\alpha_g(ae)-ae|)\to 0$. Hence the claim follows.
\end{proof}

\begin{proposition}
The norm topology on $L^{1+\infty}(\M)$ is finer than the topology ($\emph{loc}$). 
\end{proposition}

\begin{proof}
A slight modification of the former proof will suffice. Let $e$ be a projection with finite trace. For any $a\in L^1(\M)$ and $b\in \M$, we then have that 
$$\tau(|ae|)\leq \tau(|a|)=\|a\|_1\quad \mbox{and}\quad \tau(|be|) \leq \|b\|_\infty\tau(e).$$Any element $f\in L^{1+\infty}(\M)$ can be written in the form $f=a+b$ where as 
above $a\in L^1(\M)$ and $b\in \M$. From the previous proof, we recall that $e|f|e\leq |fe|$. On using this fact, it now follows from the centred equation that 
$$\tau(e|f|e)\leq \tau(|fe|)\leq\tau(|ae|) + \tau(|be|) \leq (\|a\|_1+\|b\|_\infty)(1+\tau(e)).$$If we now take the infimum over all decomposition of the form $f=a+b$, it 
will follow from the above that $$\tau(e|f|e)\leq \|f\|_{1+\infty}(1+\tau(e)).$$From this inequality it is clear that if $f_n\to f$ in norm in $L^{1+\infty}(\M)$, then we 
will have that $\tau(e|f-f_n|e)\to 0$ for each seminorm $\tau(e|\cdot|e)$.
\end{proof}

\begin{corollary}
Let $K$ and $K_0$ be compact subsets of $G$ with $K\subset K_0$, let $W$ be the completion of $L^{1+\infty}(\M)$ equipped with the topology ($\emph{loc}$). Then for any 
$a\in L^{1+\infty}(\M)$,  the function $\chi_K\otimes_\alpha a: K_0\to L^{1+\infty}(\M):g\to \chi_K(g)\alpha_g(a)$ belongs to $L^1(K)\widetilde{\otimes}_\pi W=L^1(K;W)$. 
\end{corollary}

\begin{proof} This is a trivial consequence of the dual observations that $\chi_K\otimes_\alpha a\in L^1(K;L^{1+\infty}(\M))$, and that $L^1(K;L^{1+\infty}(\M))\equiv 
L^1(K)\otimes_\pi L^{1+\infty}(\M)$ injects into $L^1(K)\widetilde{\otimes}_\pi W$. 
\end{proof}

With all the preparation done, we are now finally ready to construct the operator $T^\#$. Let $\iota$ denote the canonical continuous embedding 
$L^{1+\infty}(\M)\hookrightarrow (L^{1+\infty}(\M),(loc))\subset W$. Since each of $L^{1+\infty}(\M)$ and $W$ is complete, we also have that 
$C(G)\widetilde{\otimes}_\epsilon F = C(G, F)$ for each of $F\in\{L^{1+\infty}(\M), W\}$. Thus for each of these spaces, point evaluation at $e\in G$ is a well-defined map 
from $C(G)\widetilde{\otimes}_\epsilon F$ to $F$. We will denote this map by $\delta_e^{(F)}$. Finally given compact subsets $K\subset K_0$ of $G$, write $j_K$ for the map 
$a\to \chi_K\otimes_\alpha a$. Given a transferable operator 
$T$, consider the following diagram: (Each arrow represents a linear map.) 

$$\begin{array}{ccccccc}
W & \overset{j_{K_0,K}}{\longrightarrow} & L^1(K_0)\widetilde{\otimes}_\pi W & \overset{T\otimes \mathrm{Id}_W}{\longrightarrow} & C(G)\widetilde{\otimes}_\epsilon W & \overset{\delta_e^{(W)}}{\longrightarrow} & W\\
& & & & & &\\
\iota\uparrow &  & (\mathrm{Id} \otimes \iota)\uparrow & & (\mathrm{Id} \otimes \iota)\uparrow &  & \iota\uparrow\\
& & & & & &\\
L^{1+\infty}(\M) & \overset{j_{K_0,K}}{\longrightarrow} & L^1(K_0)\widetilde{\otimes}_\pi L^{1+\infty} & \overset{T\otimes \mathrm{Id}_{L^{1+\infty}}}{\longrightarrow} & C(G)\widetilde{\otimes}_\epsilon L^{1+\infty} & \overset{\delta_e^{(L^{1+\infty})}}{\longrightarrow} & L^{1+\infty}(\M)
\end{array}$$
It is not difficult to verify that the above diagram commutes. For example if we consider the second ``block'', then starting with a simple tensor 
$f(t)\otimes a \equiv f(t)a$ (where $f\in L^1(K_0)$ and $a\in L^{1+\infty}(\M)$, it is not difficult to see that both of $(T\otimes \mathrm{Id}_W)(\mathrm{Id} \otimes 
\iota_{L^{1+\infty}})(f\otimes a)$ and $(\mathrm{Id} \otimes \iota)(T\otimes \mathrm{Id}_{L^{1+\infty}})(f\otimes a)$ lead to the element $T(f)\otimes a$ of $C(K_0, W)$. On 
passing to linear combinations of simple tensors and extending by continuity, it is therefore clear that $(T\otimes \mathrm{Id}_W)(\mathrm{Id} \otimes \iota_{L^{1+\infty}}) 
= (\mathrm{Id} \otimes \iota_{L^{1+\infty}})(T\otimes \mathrm{Id}_{L^{1+\infty}})$.

The commutation of the above diagram ensures that the range of the sequence of maps 
$$L^{1+\infty}(\M) \overset{j_{K_0,K}}{\longrightarrow} L^1(K_0)\widetilde{\otimes}_\pi W \overset{T\otimes \mathrm{Id}_W}{\longrightarrow}  
C(G)\widetilde{\otimes}_\epsilon W \overset{\delta_e^{(W)}}{\longrightarrow} W,$$is contained in $\iota(L^{1+\infty}(\M))$. We have therefore proved the following 

\begin{proposition}\label{exist} Let $W$ be the completion of $L^{1+\infty}(\M)$ under the topology of bilateral local integrability. Given compact subsets $K\subset K_0$ of 
$G$, the sequence of maps $T^\#_{K_0,K}=\delta_e^{(W)}\circ (T\otimes \mathrm{Id}_W)\circ j_{K_0,K}$ is a well-defined linear operator from $L^{1+\infty}(\M)$ to 
$L^{1+\infty}(\M)$. In the case that $K_0=K$, we will simply write $T^\#_{K_0}$ for $T^\#_{K_0,K_0}$.
\end{proposition}

\begin{lemma} Let $W$ be as before. Given compact subsets $K_0\subset K_1$ of $G$, it follows that $T^\#_{K_0} = T^\#_{K_1,K_0}$.
\end{lemma}

\begin{proof} Let $K_0$ and $K_1$ be as in the hypothesis, and let $\iota_{0,1}$ denote the isometric injection $L^1(K_0)\to L^1(K_1)$ defined by extending each $f\in 
L^1(K_0)$ to a function assuming the value $0$ on $K_1\backslash K_0$. Now consider the following diagram:
$$\begin{array}{ccccccc}
W & \overset{j_{K_1,K_0}}{\longrightarrow} & L^1(K_1)\widetilde{\otimes}_\pi W & \overset{T\otimes \mathrm{Id}_W}{\longrightarrow} & C(G)\widetilde{\otimes}_\epsilon W & \overset{\delta_e^{(W)}}{\longrightarrow} & W\\
& & & & & &\\
Id_W\uparrow &  & (\iota_{0,1} \otimes \mathrm{Id}_W)\uparrow & & (\iota_{0,1} \otimes \mathrm{Id}_W)\uparrow &  & \mathrm{Id}_W\uparrow\\
& & & & & &\\
W & \overset{j_{K_0,K_0}}{\longrightarrow} & L^1(K_0)\widetilde{\otimes}_\pi W & \overset{T\otimes \mathrm{Id}_W}{\longrightarrow} & C(G)\widetilde{\otimes}_\epsilon W & \overset{\delta_e^{(W)}}{\longrightarrow} & W
\end{array}$$
As before we can prove that the diagram commutes by for each ``block'' checking the commutation on simple tensors, and then extending by continuity. The commutation of this 
diagram ensures that $T^\#_{K_1,K_0}=T^\#_{K_0}$.
\end{proof}

\begin{proposition}\label{Thashlim} Let $T:L^{loc}(G)\to C(G)$ be a transferable operator, $E$ be a compact neighbourhood of $e\in G$ and $U$ the neighbourhood guaranteed by 
semilocality. For any two compact sets $K_1$ and $K_2$ containing $U^{-1}E$, we have that $T^\#_{K_1}=T^\#_{K_2}$.
\end{proposition}

\begin{proof}
We start by analysing the effect of semilocality on the action of $T$, and then bootstrap to $T\otimes \mathrm{Id}_W$. Let $E$ be a compact neighbourhood of $e\in G$. For any 
compact set $K\supset U^{-1}E$, the support of $T(\chi_{K^c})$ will be contained in $K^cU$. On arguing as in \cite[Lemma 2.15]{dBL}, it follows that $K^cU\cap E$ is empty, 
and hence that $\delta_e(T(\chi_{K^c}))=0$. Now let $K_0$ be a fixed compact subset of $G$ with $K_0\supset U^{-1}E$, and let $K$ be any other compact set with 
$K_0 \supset K\supset U^{-1}E$. It follows from what we have just proved that $\delta_e(T(\chi_{K})) = \delta_e(T(\chi_{K_0}))$. Now pass to the operator 
$T\otimes \mathrm{Id}_W: L^1(K_0)\widetilde{\otimes}_\pi W \to C(G)\widetilde{\otimes}_\epsilon W$. For any simple tensor $f\otimes a$ (with $f\in L^1(K_0)$ and $a\in W$) 
direct computation therefore shows that $\delta_e^{(W)}(T\otimes \mathrm{Id}_W)(f\chi_K \otimes a)= \delta_e^{(W)}(T\otimes \mathrm{Id}_W)(f \otimes a)$. Next select a net  
$\{P_n\}$ of linear combinations of simple tensors converging to $\chi_{K_0}\otimes_\alpha a$ in $L^1(K_0)\widetilde{\otimes}_\pi W$. If say 
$P_n =\sum_{k=1}^m\gamma_k (f_k \otimes a_k)$ we may make sense of $(\chi_{K}\otimes \mathds{1})P_n$ by identifying it with $\sum_{k=1}^m\gamma_k (\chi_{K}f_k \otimes a_k)$. 
Observe that $\{(\chi_{K}\otimes \mathds{1})P_n\}$ is then still a net of simple tensors. For the sake of simplicity we will in the following simply write $P_n(g)$ for 
$\sum_{k=1}^m\gamma_k (f_k(g)a_k)$. Similary  $[(\chi_{K}\otimes \mathds{1})P_n](g)$ means $\sum_{k=1}^m\gamma_k (\chi_K(g)f_k(g)a_k)$. Now let 
$\widetilde{p}_\alpha=\int_{K_0}p_\alpha(\cdot)\,dh$ be the canonical seminorms determining the topology on $L^1(K_0)\widetilde{\otimes}_\pi W$ (where the $p_\alpha$'s are 
the seminorms determining the topology on $W$). The convergence of the net $\{P_n\}$ to $\chi_{K_0}\otimes_\alpha a\equiv\chi_{K_0}(\cdot)\alpha_{(\cdot)}(a)$, is the same as 
saying that $$\widetilde{p}_\alpha((\chi_{K_0}\otimes_\alpha a)-P_n)= \int_{K_0}p_\alpha(\alpha_g(a)-p_n(g))\,dh(g)\to 0\mbox{ for every }\alpha.$$However it is easy to check 
that 
\begin{eqnarray*}
\widetilde{p}_\alpha((\chi_{K}\otimes_\alpha a)-(\chi_{K}\otimes \mathds{1})P_n) 
&=& \int_{K_0}p_\alpha(\chi_{K}(g)\alpha_g(a)-[(\chi_{K}\otimes \mathds{1})P_n](g))\,dh(g)\\
&=& \int_{K}p_\alpha(\alpha_g(a)-p_n(g))\,dh(g) \\
&\leq& \int_{K_0}p_\alpha(\alpha_g(a)-p_n(g))\,dh(g).
\end{eqnarray*}
Hence the net $\{(\chi_{K}\otimes \mathds{1})P_n\}$ converges to 
$\chi_{K}\otimes_\alpha a$. From what we showed earlier for simple tensors, it is clear that $\delta_e^{(W)}(T\otimes \mathrm{Id}_W)(P_n)= \delta_e^{(W)}(T\otimes 
\mathrm{Id}_W)((\chi_{K}\otimes \mathds{1})P_n)$ for each $n$. Hence for any $a\in L^{1+\infty}(\M)$ we have by continuity that $T^\#_{K_0}(a)=\delta_e^{(W)}(T\otimes 
\mathrm{Id}_W)(\chi_{K_0}\otimes_\alpha a)= \delta_e^{(W)}(T\otimes \mathrm{Id}_W)(\chi_{K}\otimes_\alpha a)=T^\#_{K_0,K}$. On setting $K=K_1\cap K_0$, it now follows from 
the above and the lemma that  $T^\#_{K_0}= T^\#_{K_0,K}=T^\#_K= T^\#_{K_1,K}= T^\#_{K_1}$
\end{proof}

\begin{proposition}\label{2.15} Let $T:L^{loc}(G)\to C(G)$ be a transferable operator, $E$ be a compact neighbourhood of $e\in G$ and $U$ the neighbourhood guaranteed by 
semilocality. For any compact set $K_1$ containing $U^{-1}E$, we have that $(\delta_t\otimes \mathrm{Id})(T\otimes \mathrm{Id})(\chi_{K_1}\otimes_\alpha f) = 
(\delta_t\otimes \mathrm{Id})(T\otimes \mathrm{Id})(\chi_{U^{-1}E}\otimes_\alpha f)$ for all $t\in E$.
\end{proposition}

\begin{proof}Let $K=U^{-1}E$. Let $h$ be a function with support in $K^c$. Arguing as in the proof of the preceding proposition, the support of $T(h)$ will then be contained 
in $K^cU$. Also arguing as in \cite[Lemma 2.15]{dBL}, it follows that $K^cU\cap E$ is empty, and hence that $\delta_t(T(h))=0$ for all $t\in E$. As in the preceding proof 
select a net of linear combinations of simple tensors $\{P_n\}$ such that $(\chi_{K_1}\otimes \I)P_n$ converges to $\chi_{K_1}\otimes_\alpha f$. Similarly 
$(\chi_{K_1-K}\otimes \I)P_n$ converges to $\chi_{K_1-K}\otimes_\alpha f$. From what we noted in the first part of the proof we have that 
$(\delta_t\otimes \I)(T\otimes \mathrm{Id})((\chi_{K_1-K}\otimes \I)P_n)=0$. So by continuity $(\delta_t\otimes \I)(T\otimes \mathrm{Id})(\chi_{K_1-K}\otimes_\alpha f)=0$. 
The claim follows. 
\end{proof}

For any $a\in L^{1+\infty}(\M)$ the formal skew tensor $1\otimes_\alpha a \equiv \alpha_{(\cdot)}(a)$ is a pointwise limit (with respect to elements of $G$) of the net 
$\{\chi_K\otimes_\alpha a: K\subset G, K\mbox{ compact}\}$ where the net is ordered by inclusion of compact subsets of $G$. Hence we may think of the limit of the net 
$\{T^\#_{K}(a): K\subset G, K\mbox{ compact}\}$ (where $T^\#_{K}(a)=\delta_e^{(W)}(T\otimes \mathrm{Id}_W)(\chi_{K}\otimes_\alpha a)$), as representing the required value of 
the formal prescription $$a\to 1\otimes_\alpha a \to (T\otimes \mathrm{Id})(1\otimes_\alpha a) \to \delta_e((T\otimes \mathrm{Id})(1\otimes_\alpha a)).$$As we see from the 
preceding proposition, the net $\{T^\#_{K}(a): K\subset G, K\mbox{ compact}\}$ actually does converge. Hence it makes sense to make the following definition:

\begin{definition}\label{Thashdef}
Let $T:L^{loc}(G)\to C(G)$ be a transferable operator, with $E$ and $U$ as in the preceding proposition and $K$ a compact subset of $G$ containing $U^{-1}E$. The transferred 
operator $T^\#$ on $L^{1+\infty}(\M)$ is then defined to be $T^\#=T^\#_K$.
\end{definition}

\part{Computing the weak/strong type of the transfer operator}

For the sake of clarity we briefly revise the basic strategy for computing the weak/strong type identified earlier. The key cornerstones of our strategy consists of the 
following: 
\begin{itemize}
\item The first step is to show that for a large class of transferable operators $T$ is a completely positive map for which we may use real interpolation to conclude that the 
operator $T\otimes\mathrm{Id}$ on $L^\infty(G)\overline{\otimes}\M$ induces contractive maps on the class of rearrangement invariant monotone Riesz-Fischer spaces 
$L^{\bm{\rho}}(L^\infty(G)\overline{\otimes}\M)$. 
\item The next step would be to show that for a compact neighbourhood $K$ of $e\in G$, the prescription $f\mapsto \chi_K\otimes f$ maps $f\in L^A(M)$ into $L^B(G;L^A(M))$.
\item Then we need to find conditions under which $L^B(G;L^A(M))$ embeds into $L^C(M\otimes L^\infty(G))$. 
\item Once inside $L^C(M\otimes L^\infty(G))$ it then follows from Lemma \ref{2.4} that $f\otimes \chi_K \mapsto f\otimes_\alpha \chi_K$ defines an isometry on the subspace 
of $L^C(M\otimes L^\infty(G))$ spanned by simple tensors of the form $\chi_K\otimes f$.
\item The final step would then be to appeal to a noncommutative version of \cite[Theorem 4.3]{L-QM} to extract the weak/strong type of $T^\#$.
\end{itemize}
 
\section{Noncommutative mixed norm spaces}

Noncommutative mixed norm spaces for $L^p$-spaces were first introduced in \cite{JP}. We here investigate the construction of mixed norm spaces from the algebra 
$L^\infty(\Gamma,\nu) \overline{\otimes}\M$ where $\M$ is a semi-finite von Neumann algebra equipped with a fns trace $\tau$ and $(\Gamma,\Sigma, \nu)$ a Radon measure space. 
In this very specific setting we wish to define mixed norm spaces of Orlicz spaces. In pursuing this objective we depart from the approach of Junge and Parcet \cite{JP} and 
choose to rather follow Blozinski's approach as described in \cite{Blo} which relies on the notion of a double decreasing rearrangement. Given $f\in L^\infty(\Gamma,\nu) 
\overline{\otimes}\M$ we will for any $x\in \Gamma$ have that $f(x)\in \M$, in which case we can then take the decreasing rearrangement of $f(x)$ with respect to $\tau$. If 
we do this for every $x\in \Gamma$, this yields a real-valued function $f_1(s,x) = \decra(s,f(x))$ of two variables . As in \cite{Blo}, we may next for any fixed $s\geq 0$ 
take the decreasing rearrangement of the function $f_1(s,\cdot)$ with respect to $\nu$, which produces the function $f_2(s,t)$ referred to as the double decreasing 
reararngement. For the second decreasing rearrangement to make sense, we must ensure that $f_1(x,t)$ is for each $t\geq 0$ in fact a measurable function with respect to the 
variable $x$. The measurability of $x\mapsto f(\cdot,x)$ will be important later when we define mixed normed spaces. To show that this measurability indeed holds, we need to 
elaborate on the nature of the functions in question. We begin by reviewing some concepts and results from \cite[\S IV.7]{Tak}. 

\begin{definition}\cite[Definition IV.7.1]{Tak}\label{muDefn} Let $E$ be a Banach space. An $E$-valued function $f$ is said to be measurable with respect to $\nu$, or simply 
$\nu$-measurable, if for any compact set $K$ and $\epsilon>0$ there exists a compact subset $K_0$ of $K$ such that $\nu(K - K_0) <\epsilon$ and $f$ is continuous on $K_0$.    
\end{definition}

The following theorem provides an alternative description of the notion of measurability introduced above.

\begin{theorem}\cite[Theorem 4.7.2]{Tak}\label{mu-msbleThm} An $E$-valued function is $\nu$-measurable if and only if the following conditions are met.
\begin{enumerate}
    \item For each $\phi$ in $E^*$ the function $x\in \Gamma \mapsto \langle f(x),\phi\rangle\in\mathbb{C}$ is $\nu$-measurable.
    \item For any compact set $K$ there exists a compact subset $K_\infty$ of $K$ such that $\nu(K - K_\infty) =0$ and $f(K_\infty)$ is separable.
\end{enumerate}    
\end{theorem}

\begin{theorem}\cite[Theorem IV.7.5]{Tak}\label{OpMu1} If $f(\cdot)$ is a $B(H)$-valued function on $\Gamma$ such that for each $\zeta\in H$ the function 
$x\in \Gamma\mapsto f(x)\zeta\in H$ is $\nu$-measurable, then for any $H$-valued $\nu$ measurable function $\zeta(\cdot)$ the $H$-valued function $x\mapsto f(x)\zeta(x)$ is 
also $\nu$-measurable.    
\end{theorem}

For von Neumann algebras we need a notion of measurability which respects the involutive structure.

\begin{definition}\cite[Definition IV.7.7]{Tak}\label{oppmudefn} A $B(H)$-valued function $f$ on $\Gamma$ is $\nu$-measurable if both the functions $x\mapsto f(x)$ and 
$x\mapsto f(x)^*$ satisfy the assumption in Theorem $\ref{OpMu1}$.     
\end{definition}

Given a concrete von Neumann algebra $\M\subset B(H)$, it then follows that for an $\M$-valued $\nu$-measurable function $f$ the function $x\mapsto \langle 
f(x)\zeta,\xi\rangle$ is $\nu$-measurable for each $\xi\in H$. We use this observation to prove that for any fixed $t>0$, the map $x\mapsto \decra(t,f(x))$ is measurable.

\begin{proposition} Let $(\Gamma,\Sigma,\nu)$ be a Radon measure space and $f\in L^\infty(\Gamma,\nu) \overline{\otimes}\M$. Then the function given by $x\mapsto \decra(s,f(x))$ is measurable.
\end{proposition}

\begin{proof} By \cite[Theorem IV.7.17]{Tak} we can consider $f$ to be a $\M$-valued function as defined in Definition \ref{oppmudefn}. Then the functions $x\mapsto f(x)\xi$ 
are $\nu$-measurable for every $\xi$ in $H$. Then from our previous observation, the maps $x \mapsto \langle f(x)\zeta,\xi\rangle$ are measurable $\mathbb{C}$-valued 
functions. 

As is pointed out in \cite{FK}, we can express the non-commutative decreasing rearrangement as
$$\decra(s,f(x)) = \inf_{e\in \mathbb{P}(\M), \tau(1-e)\leq s}\left(\sup_{\xi\in eH, \|\xi\| = 1} \langle |f(x)|\xi,\xi\rangle \right),$$
where we used the fact that $\decra(s,f(x)) = \decra(s,|f(x)|). $

Fix $s>0$. For every $k\in \mathbb{N}$, there exists a projection $e_k$ with $\tau(\I-e_k)\leq s$ such that $0\leq \sup_{\xi\in e_kH, \|\xi\| = 1} 
\langle |f(x)|\xi,\xi\rangle-\decra(t,f(x)) < \frac{1}k.$
Similarly, for every such $e_k$ we can find a $\xi_{k,n}$ with $\xi_{n,k}\in e_kH$ and $\|\xi_{n,k}\| = 1$ such that $\sup_{\xi\in e_kH, \|\xi\| = 1} 
\langle |f(x)|\xi,\xi\rangle -\langle |f(x)|\xi_{n,k},\xi_{n,k}\rangle < \frac{1}{k}  $.

Hence we can write $$\decra(t,f(x)) = \inf_k\left(\sup_{\xi_{n,k}} \langle |f(x)|\xi_{n,k},\xi_{n,k}\rangle \right).$$

The map $x\mapsto \decra(s,f(x))$, can therefore be expressed as the infimum of the supremum of countably many $\nu$-measurable functions, and hence is itself 
$\nu$-measurable.
\end{proof}

We see from the above that in the theory described in \cite[Theorem IV.7.17]{Tak}, the function $f_1:x\mapsto \decra(s,f(x))$ is an element of the tensor product 
$L^\infty(\Gamma,\nu)\overline{\otimes} L^\infty(\mathbb{R}^+,m)$ where $m$ is the Lebesgue measure. Since each non-negative function in this space will be the supremum of a 
sequence of non-negative functions in $L^{\infty}(L^\infty(\Gamma) \overline{\otimes}\M)$, we may (as described in subsection \ref{vNA_MG}) use order structures to widen the 
scope of our enquiry to functions in $\widetilde{(L^\infty(\Gamma) \overline{\otimes}\M)}$ (the ``$\tau$-measurable operators'' associated with 
$L^\infty(\Gamma) \overline{\otimes}\M$).

\bigskip

Let $E_1$ and $E_2$ be rearrangement invariant Banach function spaces on the measure spaces $(\mathbb{R}_+,m)$ and $(\Gamma,\nu)$ respectively. Let $\| f_1^x\|_{E_1}$ be the 
$\nu$-measurable function obtained from the map $x\mapsto \|f_1(\cdot,x)\|_{E_1}$.

\begin{definition} Given a classical mixed norm space $E_1(\Gamma)[E_2(\mathbb{R}_+)]$, we define the mixed norm space $E_1(\Gamma)[E_2(\M)]$ to be the completion of the 
space of operators $f$ in $\widetilde{L^\infty(\Gamma) \overline{\otimes}\M}$ such that $f_1\in E_1(\Gamma)[E_2(\mathbb{R})]$. 
\end{definition}

We show that $\|f\|_{E_1[E_2]} = \| \| f_1^x\|_{E_2} \|_{E_1}$ defines a norm on $E_1(\Gamma)[E_2(\M)]$

\begin{lemma}\label{L2.6} The formula $\|f\|_{E_1[E_2]} = \| \| f_1^x\|_{E_2} \|_{E_1}$ where $f$ is an element of $\widetilde{L^\infty(\Gamma) \overline{\otimes}\M}$ for 
which $f_1\in E_1(\Gamma)[E_2(\mathbb{R})]$, defines a norm on $E_1(\Gamma)[E_2(\M)]$.
\end{lemma}

\begin{proof}
By continuity it is enough to prove the claim for the space of operators mentioned in hypothesis. If $\| \| f_1^x\|_{E_2} \|_{E_1}=0$, then $\| f_1^x\|_{E_2} = 0$ for almost 
every $x$. But then $f_1^x = 0$ for almost every $x$.

Now let $f$ and $g$ be functions such that $\| \| f_1^x\|_{E_2} \|_{E_1}<\infty$ and $\| \| g_1^x\|_{E_1} \|_{E_2}<\infty$. If we fix $x\in \Gamma$, then we know from the 
theory of non-commutative Banach function spaces, that $\|f(x) + g(x)\|_{E_2} \leq \|f(x)\|_{E_2} + \|g(x)\|_{E_2}$, where the norm is taken to be the norm of the 
non-commutative rearrangement invariant Banach function space $E_2(\M,\tau)$. Notice that for a fixed $x\in \Gamma$, we have $\|f(x)\|_{E_2} = \|f_1^x\|_{E_2}$. We may now 
combine these observations to prove the triangle inequality as shown below.
\begin{align*}
    \|f+g\|_{E_1[E_2]} & = \|\|(f+g)_1^x\|_{E_2}\|_{E_1} \\
                        & \leq \|\|f_1^x\|_{E_2}+\|g_1^x\|_{E_2}\|_{E_1} \\
                        & \leq  \|\|f_1^x\|_{E_2}\|_{E_1}+\|\|g_1^x\|_{E_2}\|_{E_1} \\
                        & = \|f\|_{E_1[E_2]} + \|g\|_{E_1[E_2]}.
\end{align*}
\end{proof}

\bigskip

We now have all the ingredients to prove a non-commutative version of \cite[Theorem 3]{Mal}.

\begin{theorem}\label{Mal-nc} Let $X$, $Y$, and $Z$ be Young functions, $(\Gamma, \Sigma, \nu)$ any $\sigma$-finite measure space and $\M$ any semi-finite von Neumann algebra equiped with 
an fns trace $\tau$. Then
\begin{enumerate}
\item The space $L^Z(L^\infty(\Gamma) \overline{\otimes}\M)$ embeds continuously into $L^X(\Gamma)[L^Y(\M)]$ if and only if $X(u)Y(v) \leq Z(\theta uv)$ for all $u,v\geq 0$.
\item The space $L^X(\Gamma)[L^Y(\M)]$ embeds continuously into $L^Z(L^\infty(\Gamma) \overline{\otimes}\M)$ if and only if $Z(\theta uv) \leq X(u)Y(v)$ for all $u,v\geq 0$.
\end{enumerate}    
\end{theorem}

\begin{proof} 
For the ``only if'' direction, suppose that the embeddings hold and are continuous for any $\sigma$-finite measure space $(\Gamma,\nu)$ and von Neumann algebra $\M$. Then the 
embeddings would in particular hold for the abelian von Neumann algebra $L^\infty(\mathbb{R}_+,m)$, in which case the result reduces to the commutative case as described in 
\cite[Theorem 2]{Mal}, with the stated relationships holding between the Young functions $X$, $Y$, and $Z$.

We pass to the converse implications. For convenience we let $\N = L^\infty(\Gamma) \overline{\otimes}\M$.

\medskip

(1) Suppose that $X(u)Y(v) \leq Z(\theta uv)$ for all $u,v\geq 0$. Then by \cite[Theorem 3]{Mal} the commutative space $L^Z(\Gamma\times \mathbb{R}_+)$ embeds continuously 
into $L^X(\Gamma)[L^Y(\mathbb{R}_+)]$. 

Let $f\in L^Z(\N)$, that is to say there exists an $\alpha>0$ such that $\tau_\N(Z(\alpha |f|)) <\infty$, where $\tau_\N$ is the canonical trace $\tau_\N = \int_G\cdot d\nu 
\otimes \tau$ on $\N$. Then 
\begin{align*}
   \tau_\N(Z(\alpha|f|)) & = \int_G\tau(Z(\alpha|f|(g)))dg \\
                        & = \int_G\int_0^\infty \decra\left(s,Z(\alpha|f|(g))\right)\;ds \; dg \\
                        & =  \int_G\int_0^\infty Z(\decra(s,\alpha|f|(g)))\;ds \; dg \\
                        & =  \int_G\int_0^\infty Z(\alpha\decra(s,|f|(g)))\;ds \; dg \\
                        & = \int_G\int_0^\infty Z(\alpha f_1(s,g))\;ds \; dg.
\end{align*}
So we have $f_1\in L^Z(\Gamma\times \mathbb{R}_+)$, which by the commutative result implies $f_1\in L^X(\Gamma)[L^Y(\mathbb{R}_+)]$, whence $f\in L^X(\Gamma)[L^Y(\M)]$ by 
definition.

\medskip

(2) Now suppose that $Z(\theta uv) \leq X(u)Y(v)$. Then \cite[Theorem 3]{Mal} ensures that $L^X(\Gamma)[L^Y(\mathbb{R}_+)]$ embeds continuously into $L^Z(\Gamma\times 
\mathbb{R}_+)$. 

If $f\in L^X(\M)[L^Y(G)]$, then $f_1\in L^X(G)[L^Y({R}_+)]$, in which case we have $f_1 \in L^Z(\Gamma\times \mathbb{R}_+)$ by the commutative result. Then there exists 
$\alpha>0$ such that $\int_\Gamma\int_0^\infty Z(\alpha f_1(s,g))\;ds\;dg<\infty$. But then 
\begin{align*}
    \int_\Gamma\int_0^\infty Z(\alpha f_1(s,g))\;ds\;d\nu(g) & = \int_\Gamma\int_0^\infty Z(\decra(s,\alpha|f|(g)))\;ds\;d\nu(g) \\
& = \int_\Gamma\int_0^\infty Z(\alpha \decra(s,|f|(g)))\;ds\;d\nu(g) \\                
& = \int_\Gamma\tau(Z(\alpha|f|(g))\;d\nu(g). 
\end{align*}

Hence $\int_G\tau(Z(\alpha|f|(g))\;dg<\infty$ and so $f\in L^Z(\N)$.

The continuity of the embeddings follows from the continuity of the embeddings in the commutative case.
\end{proof}

The above facts yield the following consequence. We actually believe this claim to be true for all spaces of the form $L^X(\Gamma)[L^Y(\M)]$, but are at present not yet 
able to prove this more general claim.

\begin{proposition} Let $X$, $Y$, and $Z$ be Young functions, $(\Gamma, \Sigma, \nu)$ any $\sigma$-finite measure space and $\M$ any semi-finite von Neumann algebra equiped 
with a fns trace $\tau$. If $Z(\theta uv) \leq X(u)Y(v)$ for all $u,v\geq 0$, then every element of $L^X(\Gamma)[L^Y(\M)]$ consists of an operator 
$f\in\widetilde{L^\infty(\Gamma) \overline{\otimes}\M}$ for which $f_1\in L^X(\Gamma)[L^Y(\mathbb{R})]$.
\end{proposition}

\begin{proof} This follows from the fact that then $L^X(\Gamma)[L^Y(\M)]\subset L^Z(L^\infty(\Gamma) \overline{\otimes}\M)\subset \widetilde{(L^\infty(\Gamma) 
\overline{\otimes}\M)}$.
\end{proof}

If $\M$ acts on the Hilbert space $H$, then by \cite[Proposition IV.7.14]{Tak} the space $H\overline\otimes L^2(\Gamma)$ can be identified with the Hilbert space 
$L^2(\Gamma;H)$ of $H$ valued functions such that $\|\zeta\|_2 = \left(\int_\Gamma\|\zeta(x)\|_H^2dx \right)^{\frac{1}{2}} <\infty$ with inner product
$$\langle \zeta, \eta\rangle = \int_\Gamma\langle \zeta(x),\eta(x)\rangle dx,$$where $\zeta,\eta:\Gamma \mapsto H$. So the action of say 
$f\in L^{1+\infty}(L^\infty(\Gamma) \overline{\otimes}\M)$ on $L^2_H(\Gamma)$ is given by $(f\zeta)(x) = f(x)\zeta(x)$ for every $\zeta$ in the domain of $f$.

\begin{lemma}
If $f\in \widetilde{(L^\infty(\Gamma) \overline{\otimes}\M)}$ is positive, then $f(x)\geq 0$ for almost every $x$.
\end{lemma}

\begin{proof}
We will show this for a bounded operator. By considering dense subspaces of $L^2_H(\Gamma)$, the proof remains largely unchanged for densely defined closed operators in 
$\widetilde{(L^\infty(\Gamma) \overline{\otimes}\M)}$.

If $f\geq 0$, then $\langle f\zeta,\zeta\rangle \geq 0$ for all $\zeta\in L^2_H(\Gamma)$. That is to say that $\int_\Gamma\langle f(x)\zeta(x),\zeta(x)\rangle dx \geq 0$. 
Suppose there exists a compact neighbourhood $K\subset \Gamma$ such that $\langle f(x)\zeta(x),\zeta(x)\rangle <0$ for all $\zeta\in L^2_H(K)$. But then 
$\zeta \chi_K \in L^2_H(\Gamma)$ with  $0>\int_K\langle f(x)\zeta(x),\zeta(x)\rangle dx =  \int_\Gamma\langle f(x)\zeta(x)\chi_K(x),\zeta(x)\chi_K(x)\rangle dx$, a 
contradiction. Hence $\langle f(x)\zeta(x),\zeta(x)\rangle >0$. Since this is true for any $\zeta \in L^2_H(\Gamma)$, in particular we have that 
$\langle f(x)\zeta,\zeta\rangle>0$ for every $\zeta\in H$. Hence $f(x)>0$ almost everywhere.
    
\end{proof}

\begin{corollary} \label{spectral order}
    If $f\leq g$ in $\widetilde{(L^\infty(\Gamma) \overline{\otimes}\M)}$, then $f(x)\leq g(x)$ almost everywhere.
\end{corollary}

The proposition below reconfirms a claim we noted earlier in the synopsis in subsection \ref{vNA_MG}.

\begin{proposition}\label{ptwise-spectral}
For evere positive element $f$ of $\widetilde{(L^\infty(\Gamma) \overline{\otimes}\M)}$ spectral projections $\chi_{[0,t]}(f)$ are pointwise the spectral projections 0f the 
associated $f(x)$s. That is the spectral projections of $f$ satisfy the property that $\chi_{[0,t]}(f(x)) = \chi_{[0,t]}(f)(x)$ for almost all $x$.
\end{proposition}

\begin{proof}
In order to prove the proposition, we must show that $\chi_{[0,t]}(f)(x)$ satisfies all the properties of the spectral projections of $f(x)$ and exploit their uniqueness to 
reach our conclusion. 

\begin{enumerate}
\item $\chi_{[0,t]}(f) \nearrow \I$ as $t\nearrow \infty$, and hence $\chi_{[0,t]}(f)(x) \nearrow \I$ almost everywhere. Similarly $\chi_{[0,t]}(f) =0$ for all $t<0$, hence 
$\chi_{[0,t]}(f)(x) =0$ almost everywhere for all $t<0$.
\item If $t_1 \leq t_2$, then $\chi_{[0,t_1]}(f) \leq \chi_{[0,t_2]}(f)$, and hence $\chi_{[0,t_1]}(f)(x) \leq \chi_{[0,t_2]}(f)(x)$ almost everywhere by Corollary 
\ref{spectral order}.
\item $\chi_{[0,t+\epsilon]}(f) \searrow \chi_{[0,t]}(f)$ as $\epsilon\searrow 0$ and hence $\chi_{[0,t+\epsilon]}(f)(x) \searrow \chi_{[0,t]}(f)(x)$ almost everywhere.
\item Since $f\chi_{[0,t]}(f) \leq t\chi_{[0,t]}(f)$ and $f\chi_{(t,\infty)}(f)\geq t\chi_{(t,\infty)}(f)$, we have that $f(x)\chi_{[0,t]}(f)(x) \leq t\chi_{[0,t]}(f)(x)$ and 
$f(x)\chi_{(t,\infty)}(f)(x)\geq t\chi_{(t,\infty)}(f)(x)$ almost everywhere.
\end{enumerate}

Hence for almost all $x$, $\chi_{[0,t]}(f)(x)$ are the spectral projections of $f(x)$. 
    
\end{proof}

We can exploit this fact in the following proposition. 

\begin{proposition} For any $f\in \widetilde{(L^\infty(\Gamma) \overline{\otimes}\M)}$ the distribution functions of $f$ and $\decra(\cdot,f(\cdot)):(s,x) \mapsto 
\decra(s,f(x))$ satisfy $d_f(\theta) = d_{\decra(\cdot,f(\cdot))}(\theta)$.

\end{proposition}

\begin{proof}
We denote the measure on $\Gamma$ by $\nu$ and the Lebesgue measure by $m$. Let $e=\chi_{(\theta,\infty)}(|f|)$ and $E = \{(s,x)\in \mathbb{R}_+\times \Gamma: 
\decra(s,f(x))>\theta\}$. For $x\in \Gamma$ let $e_x = \chi_{(\theta,\infty)}(|f(x)|)$ and $E_x = \{s\in\mathbb{R}_+: (s,x)\in E\}$. Note that in this notation, $e(x) = e_x$ 
almost everywhere.
    
We claim that $d_{\decra(\cdot,f(x))}(\theta) = d_{f(x)}(\theta)$. This is a consequence of the observations that for any $a\in \tM$ and any $t>0$ we have that 
$\decra(\cdot, \chi_{(t,\infty)}(|a|)) = \chi_{(t,\infty)}(\decra(\cdot, |a|)) = \chi_{(t,\infty)}(\decra(\cdot, |a|))$ and that 
$\int_0^\infty\decra(s, \chi_{(t,\infty)}(|a|))\,ds =\tau(\chi_{(t,\infty)}(|a|))$. The first claim in turn follows by first noting that the continuous functions 
$$\gamma_n(s) = \left\{\begin{array}{ll} 0 & 0\leq s \leq t\\ ns-nt & t<s\leq t+\tfrac{1}{n}\\ 1 & t+\tfrac{1}{n} < s\end{array}\right.$$increase to 
$\chi_{(t,\infty)}$ and then applying \cite[Proposition 5.12(iii)(vi)]{GLbook}. Having noted this, it is now clear that 
$m(E_x) = d_{\decra(\cdot,f(x))}(\theta) = d_{f(x)}(\theta) = \tau(e_x)$. Furthermore
\begin{align*}
    d_f(\theta)  & =  \tau(e) \\
                        & = \int_\Omega\tau(e(x))dx \\
                        & = \int_\Omega\tau(e_x)dx \\
                        & =  \int_\Omega d_{f(x)}(\theta)dx \\
                        & =\int_\Omega\lambda(E_x)dx \\
                        & = (m\times\nu)(E) = (m\times\nu)(E) \\
                        & = d_{\decra(\cdot,f(\cdot))}(\theta).
\end{align*}

\end{proof}

The proof of the following theorem is ostensibly identical to the proof of \cite[Theorem 3.12]{Blo}.
\begin{theorem}
If $E_1$ and $E_2$ are rearrangement invariant Banach function spaces of measurable functions on $\mathbb{R}_+$ and $\Gamma$ respectively with K\"othe duals $E_1'$ and 
$E_2'$, then $E_1'(\Gamma)[E_2'(\M)] = (E_1(\Gamma)[E_2(\M)])'$. 
\end{theorem}    

\begin{proof} We show that the spaces are canonically isometric. Suppose that $f\in E_1'(\Gamma)[E_2'(\M)]$. Then by \cite[Theorem 3.9.10]{DdPS},
\begin{align*}
    \|f\|_{(E_1(\M)[E_2(\Gamma)])'} & = \sup_{\|g\|_{E_1[E_2]}\leq 1} \tau(|fg|)\\
                       & = \sup_{\|g\|_{E_1[E_2]}\leq 1}\int_\Gamma\int_0^\infty\decra(s,f(x)g(x))dsdx \\
                       & \leq \sup_{\|g\|_{E_1[E_2]}\leq 1}\int_\Gamma\int_0^\infty\decra(s,f(x))\decra(s,g(x))dsdx \\
                       & \leq \sup_{\|g\|_{E_1[E_2]}\leq 1}\int_\Gamma\|f(x)\|_{E_1'}\|g(x)\|_{E_1}dx\\
                      & \leq \sup_{\|g\|_{E_1[E_2]}\leq 1}\big\|\|f(x)\|_{E_1'}\big\|_{E_2'}\big\|\|g(x)\|_{E_1}\big\|_{E_2} \\
                      &\leq \|f(x)\|_{E_1'[E_2']}<\infty.
\end{align*}
Hence $f\in (E_1(\M)[E_2(\Gamma)])'$. 

Now let $f\in(E_1(\M)[E_2(\Gamma)])'$. Let $P = p\otimes \chi_I$ for $p\in\mathbb{P}(\M)$ with finite trace and $I\subset \Omega$ with finite measure, then

$$\big\|\|fP\|_{E_1'}\|_{E_2'} = \sup_h\int_I\sup_g\int_0^\infty \decra(s,f(x)P(x))\decra(s,g(x))ds\; h(x)dx$$where $g$ and $h$ are functions such that $g\in P\N P$, the 
support of $h$ is in $I$, $\|g(x)\|_{E1}\leq 1$ for each $x$ and $\|h\|_{E_2}\leq 1$. For every $\epsilon>0$ we can find $g\in\widetilde{\N}$ and an almost everywhere finite 
measurable function $h$ on $\Gamma$ such that the above is majorised by 
$$\int_I\int_0^\infty \decra(s,f(x)P(x))\decra(s,g(x))h(x)\; ds\; dt + \int_I\epsilon h(t) dt+\epsilon.$$However, as $\|g(x)h(x)\|_{E_1[E_2]}\leq 1,$ this is in turn 
majorised by $$\sup_{\|k(x)\|_{E_1[E_2]}\leq 1}\int_\Gamma\int_0^\infty \decra(s,f(x)P(x))\decra(s,k(x))ds \; dx + \epsilon \|h\chi_I\|_1+\epsilon.$$
Now since the restriction of $E_2$ to $I$ is contained in $L^1(I),$ we can find $K>0$ dependent on $E_2$ and the projection $P$ such that the last expression is majorised by 
$$\|fP\|_{(E_1(\M)[E_2(\Gamma)])'} + \epsilon K\|h\|_{E_2} + \epsilon \leq \|fP\|_{(E_1(\M)[E_2(\Gamma)])'} + \epsilon K +\epsilon.$$

If we now first let $\epsilon \rightarrow 0$, after which we then let $P \uparrow \I_\N$ by letting $p\uparrow \I_\M$ and $I\uparrow \Gamma$, we get that 
$\|f\|_{E_1'[E_2']}\leq \|f\|_{(E_1(\M)[E_2(\Gamma)])'}$, completing the proof.
\end{proof}

\section{Comparing $\widetilde{T}$ and $T^\#$}

As far as lifting the strategy of \cite{L-QM} to the present setting is concerned, Corollary \ref{4.6c} is the most crucial result of this section. We pause to note the 
following fact which will be silently used throughout this section. Given $f\in L^{1+\infty}(G)$ Proposition \ref{Thashlim} ensures that for a large enough compact 
neighbourhood $K$ of $e\in G$ we have that $T^\#(f)=\delta_e((T\otimes \mathrm{Id})(\chi_K\otimes_\alpha f))$. We start our analysis by noting an easy corollary to Lemma 
\ref{2.4}.

\begin{corollary} 
Let $X$ be a Young function. For every $h\in L^{1\cap\infty}(G)$ and $f\in L^{1+\infty}(\M, \tau)$ we have that $\|h\otimes_\alpha f\|_X =\|h\otimes f\|_X$. Follows from 
Lemma \ref{2.4}
\end{corollary}

\begin{proof}
This actually is actually a trivial consequence of \cite[Proposition IV.7.16]{Tak}.
\end{proof}

\begin{lemma}[Proposition 7.3; \cite{GLbook}]\label{lem2p22} If $\varphi$ is a fundamental function, then 
$\displaystyle\sup_{t>0}\decra(t,f)\varphi(t)=\sup_{s>0}s\varphi(d_f(s))$.\end{lemma}

\begin{lemma}\label{4.6a} Let $T:L^{loc}(G)\to C(G)$ be a transferable operator, $E$ be a compact neighbourhood of $e\in G$ and $U$ the neighbourhood guaranteed by 
semilocality. For any $s\in E$ we have that  $(\delta_s\otimes \mathrm{Id})[(T\otimes \mathrm{Id})(\chi_K\otimes_\alpha f)]=\alpha_s(T^\#(f))$.
\end{lemma}

\begin{proof} First note that in the required equality, we may use Proposition \ref{2.15} to replace $U^{-1}E$ by an arbitrarily large compact superset $K$ of $U^{-1}E$. With 
this in mind let $s\in E$ and $f\in L^{1+\infty}(\M)$ be given, and select a net of elements of the form $\sum_{j=1}^n(\chi_{B_j}\otimes \alpha_{g_j}(f))$ converging 
pointwise ae to $(\chi_K\otimes_\alpha f)$, where $B_j$ ($1\leq j\leq n$) are disjoint measurable sets such that $\cup_{j=1}^n B_J = K$. Notice that for any $t,s \in G$, 
$st\in K\Leftrightarrow t\in s^{-1}K$. Thus given $s,t\in G$, we have $(\lambda_s\otimes \mathrm{Id})(\sum_{j=1}^n(\chi_{B_j}\otimes 
\alpha_{g_j}(f)))(t)=\sum_{j=1}^n(\chi_{B_j}(st)\alpha_{g_j}(f))\to \chi_K(st)\alpha_{st}(f)=\chi_{s^{-1}K}(t)\alpha_{s}(\alpha_t(f))$. 

Next using a similar approximation, one can show that for any fixed compact neighbourhood $K_0$, we have that $$(\mathrm{Id}\otimes \alpha_s)(\chi_{K_0}\otimes_\alpha f)(t)= 
\chi_{K_0}(t)\alpha_s(\alpha_t(f)).$$Thus by continuity $$(\lambda_s\otimes \mathrm{Id})(\chi_K\otimes_\alpha f)(t)= (\mathrm{Id}\otimes 
\alpha_s)(\chi_{s^{-1}K}\otimes_\alpha f)(t).$$

If $K$ is a large enough neighbourhood of $e\in G$, we will have that $(\delta_e\otimes \mathrm{Id})(T\otimes \mathrm{Id})(\chi_{s^{-1}K}\otimes_\alpha f)=T^\#(f)$. Now 
recall that $T$ and $\lambda_s$ commute. Using this fact and hoping for the best, we then formally have that 
\begin{eqnarray*}(\delta_s\otimes \mathrm{Id})(T\otimes \mathrm{Id})(\chi_K\otimes_\alpha f) &=& (\delta_e\otimes \mathrm{Id}) (\lambda_s\otimes \mathrm{Id})(T\otimes \mathrm{Id})(\chi_K\otimes_\alpha f)\\
&=&(\delta_e\otimes \mathrm{Id}) (T\otimes \mathrm{Id})(\lambda_s\otimes \mathrm{Id})(\chi_K\otimes_\alpha f)\\
&=& (\delta_e\otimes \mathrm{Id}) (T\otimes \mathrm{Id})(\mathrm{Id}\otimes \alpha_s)(\chi_{s^{-1}K}\otimes_\alpha f)\\
&=&(\mathrm{Id}\otimes \alpha_s)(\delta_e\otimes \mathrm{Id}) (T\otimes \mathrm{Id})(\chi_{s^{-1}K}\otimes_\alpha f)\\
&=& \alpha_s(T^\#(f)).
\end{eqnarray*}
Hence the claim follows.
\end{proof}

\begin{proposition}[\cite{dBL}, inequality (4.5)]\label{dBL_4.5} Let $T:L^{loc}(G)\to C(G)$ be a transferable operator, $K$ be a compact neighbourhood of $e\in G$ and $U$ the 
neighbourhood guaranteed by semilocality. Given $f\in L^{1+\infty}(\M)$, let $\overline{p}=\chi_{(\lambda,\infty)}(\widetilde{T}(\chi_{U^{-1}K}\otimes_\alpha f))$ 
and $p = \chi_{(\lambda,\infty)}(T^\#(f))$. Then $(\int_G\otimes\tau)(\overline{e})\geq \tau(p)|K|$.
\end{proposition}

\begin{proof} Taking into account the comments in subsection \ref{vNA_MG} regarding the Borel functional calculus and projections in $\M\overline{\otimes}L^\infty(G)$, we 
will for each $s\in K$ have that, 
\begin{eqnarray*}
\overline{p}(s) &=& \chi_{(\lambda,\infty)}(|(\delta_s\otimes \mathrm{Id})[(T\otimes \mathrm{Id})(\chi_{U^{-1}K}\otimes_\alpha f)]|)\\
&=& \chi_{(\lambda,\infty)}(|\alpha_s(T^\#(f))|)\\
&=& \alpha_s(\chi_{(\lambda,\infty)}(|T^\#(f)|))\\
&=& \alpha_s(p).
\end{eqnarray*}
But then $$\left(\int_G\otimes\tau\right)(\overline{p})= \int_G\tau(\overline{p}(s))\,ds\geq \int_K\tau(\overline{p}(s))\,ds =\int_K\tau(\alpha_s(p))\,ds = \int_K\tau(p)\,ds = |K|\tau(p).$$
\end{proof}

The above proposition now yields following corollary which we may use to estimate the norm of $T^\#(f)$ in terms of that of $\widetilde{T}(\chi_{U^{-1}K}\otimes_\alpha f)$.

\begin{corollary}[dBL Lemma 4.6]\label{4.6c}
Let $T:L^{1+\infty}(G)\to C(G)$ be a transferable operator, $K$ a compact neighbourhood of $e\in G$ and $U$ the neighbourhood guaranteed 
by semilocality. Given $f\in L^{1+\infty}(\M,\tau)$, we have that $\decra(t,\widetilde{T}(\chi_{U^{-1}K}\otimes_\alpha f))\geq 
\decra(t/|K|,T^\#(f))$ for each $t\geq 0$.
\end{corollary}

\begin{proof} The proof follows fairly directly from the preceding result and \cite[Proposition 2.2]{FK}.
\end{proof}

\section{Norm estimates for $T^\#$}

In addition to the estimates afforded by Corollary \ref{4.6c}, we shall also need a noncommutative version version of the following theorem of O'Neil.

\begin{theorem}[\cite{ONeil}, Theorem 8.15]\label{8.15} Let $(X,\nu_X)$ and $(Y,\nu_y)$ be general measurable space and $f$ and $g$ measurable functions on $X$ and $Y$ 
respectively, with $h=f\otimes g$. Let $A$, $B$, $C$ be Young functions for which there is some $\theta>0$ satisfying for which $A^{-1}(u)B^{-1}(v)\leq \theta C^{-1}(uv)$ for 
all $u,v\geq 0$  Then 
\begin{enumerate} 
\item $f\in M_A, g\in L^B \Rightarrow h\in M_C(X\times Y)$ with $\|h\|_{M_C}\leq 3\theta\|f\|_{M_A}\|g\|_{L_B}$.
\item $f\in L_A, g\in L^B \Rightarrow h\in L_C(X\times Y)$ with $\|h\|_{L_C}\leq \theta\|f\|_{L_A}\|g\|_{L_B}$ (T And\^o).
\item $f\in \Lambda_A, g\in \Lambda_B \Rightarrow h\in \Lambda_C(X\times Y)$ with $\|h\|_{\Lambda_C}\leq 2e^3\theta\|f\|_{\Lambda_A}\|g\|_{\Lambda_B}$.
\end{enumerate}
\end{theorem}

An application of the ideas presented in the discussion preceding Lemma \ref{L2.6} to the above theorem now leads to the following corollary.

\begin{corollary} Let $\M$ and $G$ be as before with $g\in L^{1+\infty}(\M,\tau)$ and $f\in L^{1+\infty}(G)$ be given and set 
$h=f\otimes g$. Let $A$, $B$, $C$ be Young functions for which there is some $\theta>0$ satisfying for which $A^{-1}(u)B^{-1}(v)\leq \theta C^{-1}(uv)$ for all $u,v\geq 0$  
Then 
\begin{enumerate} 
\item $f\in M_A(\M,\tau), g\in L^B(G) \Rightarrow h\in M_C(L^\infty(G)\overline{\otimes}\M)$ with $\|h\|_{M_C}\leq 3\theta\|f\|_{M_A}\|g\|_{L_B}$.
\item $f\in L_A(\M,\tau), g\in L^B(G) \Rightarrow h\in L_C(L^\infty(G)\overline{\otimes}\M)$ with $\|h\|_{L_C}\leq \theta\|f\|_{L_A}\|g\|_{L_B}$.
\item $f\in \Lambda_A(\M,\tau), g\in \Lambda_B(G) \Rightarrow h\in \Lambda_C(L^\infty(G)\overline{\otimes}\M)$ with 
$\|h\|_{\Lambda_C}\leq 2e^3\theta\|f\|_{\Lambda_A}\|g\|_{\Lambda_B}$.
\end{enumerate}
\end{corollary}

\begin{remark} 
Showing that \cite[Theorem 8.15]{ONeil} extends to the present context enables us to provide criteria for showing that for example $f\in L^A(\M,\tau)$ implies that 
$\chi_K\otimes f\in L^C(L^\infty(G)\overline{\otimes}\M)$ for each compact $K\subset G$. To show this we may use the theory of double decreasing rearrangements for von 
Neumann algebras of the form $L^\infty(G)\overline{\otimes}\M$. 
\end{remark}

If \cite[Lemmata 2.4 \& 4.6]{dBL} are replaced by Lemmata \ref{2.4} and \ref{4.6c}, then for the proofs of Theorems 4.1 and 4.3 and Lemma 4.2 of \cite{L-QM} to go through in 
the present context, all that is needed is the version of \cite[Theorem 8.15]{ONeil} mentioned above. Leaving the proofs as an exercise, we therefore merely cite the quantum 
versions of these results.

\begin{theorem}\label{Ttildecont} Let $\M$ be a semifinite von Neumann algebra equipped with an fns trace and admitting a group action of some locally compact group $G$ 
and $T$ a transferable operator from $L^{loc}(G)$ to $C(G)$. Suppose that $B(u)C(u)\leq A(\theta uv)$ for all $u,v\geq 0$. 
Then for any $f\in L^{1+\infty}(\M,\tau)$ the following holds: 
\begin{enumerate}
\item $\widetilde{T}(\chi_{KU^{-1}}\otimes_\alpha f)\in L^A(L^\infty(G)\overline{\otimes}\M,\int_G\otimes\tau)\ \Rightarrow\ 
\varphi_C(|K|)\|T^\#(f)\|_B\leq\theta\|\widetilde{T}(\chi_{KU^{-1}}\otimes_\alpha f)\|_A$;
\item $\widetilde{T}(\chi_{KU^{-1}}\otimes_\alpha f)\in M_A(L^\infty(G)\overline{\otimes}\M,\int_G\otimes\tau)\ \Rightarrow\ 
\varphi_C(|K|)\|T^\#(f)\|_{M_B}\leq\theta\|\widetilde{T}(\chi_{KU^{-1}}\otimes_\alpha f)\|_{M_A}$;
\end{enumerate}
\end{theorem}

\begin{lemma} Let $\M$ be as before $G$. Suppose that $\theta\varphi_X(s)\varphi_Y(t)\geq \varphi_Z(st)$ for all $t,s>0$. For any $f\in L^Y(\M,\tau)$ and any compact 
$K\subset G$, we have that $\chi_{KU^{-1}}\otimes_\alpha f\in L^Z(L^\infty(G)\overline{\otimes}\M,\int_G\otimes\tau)$ with 
$\|\chi_{KU^{-1}}\otimes_\alpha f\|_Z\leq \theta\varphi_X(|KU^{-1}|).\|f\|_Y$
\end{lemma}

\begin{theorem}\label{Thashcont} Let $\M$ and $T$ be as before. Suppose that $B(u)C(v)\leq A(\theta uv)$ for all $u,v\geq 0$, and that 
$\theta_2\varphi_X(s)\varphi_Y(t)\geq \varphi_Z(st)$ for all $t,s>0$. If $\varphi_X\leq \theta_3\varphi_C$, then the following holds:
\begin{enumerate}
\item If $\widetilde{T}$ boundedly maps $L^Z(L^\infty(G)\overline{\otimes}\M,\int_G\otimes\tau)$ into $L^A(L^\infty(G)\overline{\otimes}\M,\int_G\otimes\tau)$, then $T^\#$ 
boundedly maps $L^Y(\M,\tau)$ into $L^B(\M,\tau)$ with $\|T^\#\|\leq \theta_1\theta_2\theta_3\|\widetilde{T}\|$.
\item If $\widetilde{T}$ boundedly maps $L^Z(L^\infty(G)\overline{\otimes}\M,\int_G\otimes\tau)$ into $M_A(L^\infty(G)\overline{\otimes}\M,\int_G\otimes\tau)$, then $T^\#$ 
boundedly maps $L^Y(\M,\tau)$ into $M_B(\M,\tau)$ with $\|T^\#\|\leq \theta_1\theta_2\theta_3\|\widetilde{T}\|$.
\end{enumerate}
\end{theorem}

\part{Ergodic convergence theorems}
\section{Noncommutative maximal operators}\label{maxop}

For this final part of the paper we need a clear concept of Hardy-Littlewood style maximal operators conditioned to the non commutative setting. Since for a sequence $(f_n)$ 
in some monotone Riesz-Fischer space $L^{\bm{\rho}}(\M,\tau)$ we cannot in the absence of monotonicity of the sequence simplistically make sense of $\sup_n|f_n|$ in the 
noncommutative setting, developing a sensible theory of maximal operators in this context is far from trivial. We shall briefly sample the known theory on this point before 
offering a few suggestions of our own on this matter.

We first look at the spaces used to prove strong type maximal inequalities in the noncommutative context. Let $\N$ be a von Neumann algebra equipped with a normal semifinite 
faithful trace $\tau_{\N}$. For $0 < p \leq \infty$, the space $L^p(\N; \ell^\infty)$ consists of all sequences of operators $x = (x_n)_{n>0}$ for which there exist operators 
$a, b \in L^{2p}(\N)$ and contractions $(u_n)_{n>0}$ in $\N$ such that $$x_n = au_nb\mbox{ for all }n.$$We set $$\|x\|_{L^{p}(\N;\ell^\infty)} = \inf \|a\|_{2p}\|b\|_{2p}$$ 
where the infimum runs over all decompositions as above. Then $L^{p}(\N;\ell^\infty)$ is a Banach space for $p \geq 1$ and a quasi-Banach for $p < 1$. When $x$ consists of 
positive operators, we have a simpler description namely $$\|x\|_{L^p(\N;\ell^\infty)} = \inf\{\|a\|_p\colon 0\leq x_n\leq a\mbox{ for all } n\geq 0\}.$$This explains why 
$\|x\|_{L^p(\N;\ell^\infty)}$ plays the role of the norm of the maximal function. It is proved in \cite{JQ} that this infimum is achieved but it may depend on $p$ if $x$ 
belongs to $L^p(\N;\ell^\infty)$ for multiple values of $p$.

For weak type inequalities there is another space that is used. The space $\Lambda_{p,\infty}(\N; \ell^\infty)$ consists of all sequences of operators $x = (x_n)_{n>0}$ for 
which the quantity $$\|x\|_{\Lambda_{p,\infty}(\N; \ell^\infty)} = \sup_{\lambda>0}\inf\{\tau(1 - e)^{1/p}\colon \|ex_ne\|\leq \lambda\mbox{ for all }n\geq 0\}$$is finite 
where the infimum runs over all self-adjoint projections in $\N$. Thus, if $x \Lambda_{p,\infty}(\N; \ell^\infty)$ with $\|x\|_{\Lambda_{p,\infty}(\N; \ell^\infty)} < C$ for 
all $\lambda > 0$, then there is a projection $e \in \N$ with $$\tau(1 - e) \leq \left(\frac{C}{\lambda}\right)^p, \quad\|ex_ne\|\mbox{ for all }n\geq 0.$$The space is a 
quasi-Banach space. In section 4 of their paper Cadilhac and Ricard observe that this space may variously be described by as 
$$\Lambda_{p,\infty}(\N; \ell^\infty) = \{(x_n) \in L^p(\N)^{\mathbb{N}}\colon\|(x_n)\|_{\Lambda_{p,\infty}} = \sup_{t>0}t^{1/p}\decra((x_n))(t) < \infty\}$$where here 
$\decra((x_n))(t)$ is defined as $\decra((x_n))(t) = \inf_{\tau(1-e)\leq t}\sup_{n\geq 0} \|ex_ne\|$. (See \cite[Definition 4.1]{CR} and the discussion fgollowing it.) Taking 
into account that the fundamental function of $L^p$ is $\varphi_p(t)=t^{1/p}$, this space is therefore clearly related to the space $$M^*_{\varphi_A}(\widetilde{\N}),\tau)= 
\{x\in \widetilde{\N}\colon\|x\|_{M^*_{\varphi_A}} = \sup_{t>0}\varphi_A(t)\decra(t,x) < \infty\}$$. 

\medskip

In the classical theory the space which behaves nicely with regard to mixed norm spaces is 
$$M_{\varphi_A}(\Omega) = \{f\in M_0(\Omega)\colon\|f\|_{M_{\varphi_A}} = \sup_{t>0}\frac{\varphi_A(t)}{t}\int_0^t f^*(s)\,ds < \infty\}$$ not $M^*_{\varphi_A}(\N)$. The 
quantum analog of this spaces is of course just
$$M_{\varphi_A}(\N,\tau) = \{f\in \widetilde{\N}\colon\|f\|_{M_{\varphi_A}} = \sup_{t>0}\frac{\varphi_A(t)}{t}\int_0^t \decra(s,f)\,ds < \infty\}.$$
What we need to adequately describe weak type is a quantum space $M_{\varphi_A}(\N;\ell^\infty)$ which is constructed using the same general philosophy above, but mimicking 
the structure of the space $M_{\varphi_A}(\Omega)$. It is not clear if the approach of Cadilhac and Ricard will suffice to achieve this end, but their idea of constructing 
these spaces by using some sort of decreasing rearrangement of elements of $\tM^{\mathbb{N}}$ is a useful step forward. Using this idea we choose to rather define decreasing 
rearragements for elements $(x_n)$ of $\tM^{\mathbb{N}}$ by the prescription $\decra(t,(x_n))=\sup_n\decra(t,x_n)$ for each $t\geq 0$. Using these decreasing rearrangements 
we can then define sequential analogs of any of the noncommutative monotone Riesz-Fischer spaces discussed in \cite[\S 7.2]{GLbook}. Specifically for any such monotone 
Riesz-Fischer norm $\bm{\rho}$ we can then define the space $L^{\bm{\rho}}(\M,\ell^\infty)$ to be $L^{\bm{\rho}}(\M,\ell^\infty) = 
\{(x_n) \in\tM^{\mathbb{N}}:\decra(\cdot,(x_n))\in L^{\bm{\rho}}(\mathbb{R}_+)\}$. The function $t\mapsto \decra(t,(x_n))$ inherits many of the basic properties of decreasing 
rearrangements from those for elements of $\tM$. It is however not clear to us if $\|(x_n)\|_{\bm{\rho}} = \bm{\rho}(\decra(\cdot,(x_n)))$ will still be a norm. This is not a 
major obstacle, since the space will still be a well-behaved quasi-normed space even if this aspect fails. This can be seen by noting that the following proposition may be 
derived from \cite[Proposition 5.12]{GLbook}. (Properties (i) and (iii) require a bit of work; the rest are straightforward. We provide salient details for (i) and (iii).)

\begin{proposition}\label{P5.FK2.5} Let $(f_n)$, $(a_n)$ and $(b_n)$ be sequences of $\tau$-measurable operators.
\begin{enumerate}
\item[(i)] The map $\decra(\cdot,(f_n)):\mathbb{R}_+\to [0,\infty]:t\mapsto\decra(t,(f_n))$ is non-increasing and right-continuous. Moreover  
$\lim_{t\searrow 0}\decra(t,(f_n)) = \sup_n\|f_n\|_\infty.$
\item[(ii)] For any $t>0$ and any $\alpha\in \mathbb{C}$, $\decra(t,(f_n)) = \decra(t,(f_n^*)) = \decra(t,(|f_n|))$, and $\decra(t,\alpha(f_n)) = |\alpha|\decra(t,(f_n)).$
\item[(iii)] If $a_n\geq b_n\geq 0$ for each $n$, then  $\decra(t,(a_n))\geq\decra(t,(b_n))$ for any $t>0.$ Moreover if $(a_n^\alpha)\subseteq(\tM_+)^{\mathbb{N}}$ is a net 
increasing to $(a_n)$ in $\tM^{\mathbb{N}}$, then $\sup_\alpha\decra(t,(a^\alpha_n))=\decra(t,(a_n)).$
\item[(iv)] For any $s,t>0$, $\decra((t+s),(a_n+b_n)) \leq \decra(t,(a_n)) + \decra(s,(b_n)).$
\item[(v)] For any $s,t>0$, $\decra(t+s,(a_nb_n))\leq \decra(t,(a_n))\decra(s,(b_n)).$ We have, in particular, 
$\decra(t,(a_nf_nb_n))\leq (\sup_n\|a_n\|)(\sup_n\|b_n\|)\decra(t,(f_n)).$
\item[(vi)] Let $\Phi:[0,\infty)\to [0,\infty]$ be a non-decreasing function which is continuous on $[0,b_\Phi]$ where $b_\Phi=\sup\{t\in[0,\infty)\colon\Phi(t)<\infty\}.$ 
If $(\Phi(|f_n|))$ is again in $\tM^\mathbb{N}$, then $\Phi(\decra(t,(f_n)) = \decra(t,(\Phi(|f_n|)))$ for any $t \geq 0.$
\end{enumerate}
\end{proposition}

\begin{proof} \emph{Claim (i):} The claims regarding right-continuity deserve comment. Given $t>t_0$, the fact that for any fixed $m$ we have that 
$\decra(t,(f_n))\geq \decra(t,f_m)$, ensures that $\limsup_{t\searrow t_0}\decra(t,(f_n))\geq \lim_{t\searrow t_0}\decra(t,f_m)= \decra(t_0,f_m)$, and hence that 
$\limsup_{t\searrow t_0}\decra(t,(f_n))\geq \decra(t_0,(f_n))$. Conversely since $t\to \decra(t,(f_n))$ is non-increasing, we have that 
$\decra(t_0,(f_n))\geq\limsup_{t\searrow t_0}\decra(t,(f_n))$. Hence $\limsup_{t\searrow t_0}\decra(t,(f_n))= \decra(t_0,(f_n))$. A similar proof shows that 
$\liminf_{t\searrow t_0}\decra(t,(f_n))= \decra((f_n),t_0)$ whence $\lim_{t\searrow t_0}\decra(t,(f_n))= \decra(t_0,(f_n))$.

\emph{Claim (ii):} The fact that for any fixed $m$ and $\alpha$ we have that 
$\decra(t,(f_n^\alpha))\geq \decra(t,f_m^\alpha)$, ensures that $\sup_{\alpha}\decra(t,(f_n^\alpha))\geq \sup{\alpha}\decra(t,(f_m^\alpha))= \decra(t,f_m)$, and hence that 
$\sup_{\alpha}\decra(t,(f_n^\alpha))\geq \decra(t,(f_n))$. Conversely since for any fixed $m$ and $\alpha$ we have that $\decra(t,f_m)\geq\decra(t,f_m^\alpha)$, it is 
similarly clear that $\decra(t,(f_n))\geq\sup_\alpha\decra(t,(f_n^\alpha))$. Hence that claim holds.
\end{proof}

On the basis of the above proposition, we may therefore define sequential analogs of the Banach function spaces $L^{\bm{\rho}}(\M,\tau)$ by setting 
$L^{\bm{\rho}}(\M,\ell^\infty)=\{(x_n)\in \tM^{\mathbb{N}}: \decra(\cdot,(x_n))\in L^{\bm{\rho}}(\mathbb{R}_+)\}$. For any sequence $(T_n)$ of operators on from say 
$L^{\bm{\rho}_0}(\M,\tau)$ to $L^{\bm{\rho}_1}(\M,\tau)$ we may introduce the operator $\mathcal{T}_{max}$ given by $\mathcal{T}_{max}:L^{\bm{\rho}_0}(\M,\tau) \to 
L^{\bm{\rho}_1}(\M,\tau)^{\mathbb{N}}:x\mapsto (T_n(x))$. If $\mathcal{T}_{max}$ actually maps into $L^{\bm{\rho}_1}(\M,\ell^\infty)$, we say it is type 
$(L^{\bm{\rho}_0},L^{\bm{\rho}_1})$.

\subsection{Maximal operators in the special case where $T_n:f\mapsto k_n*f$ for all $f$} 

In the special case where we are dealing with sequences of the form $T^\#_n(x)=\int_Gk_n(t)\alpha_{t^{-1}}(x)\,dt$ where $T_n:f\mapsto k_n*f$ and $k_n\geq 0$ with 
$\|k_n\|_1\leq 1$, there is an easier alternative. Here we can set $k_{sup}=\sup_n k_n$ and write $T_{sup}$ for $f\mapsto k_{sup}*f$. In the case where for some increasing 
sequence of compact neighbourhoods $V_n$ we have that $k_n=\frac{1}{|V_n|}\chi_{V_n}$, the formula for $k_{sup}$ is particularly simple; specifically 
$$k_{sup}=\frac{1}{|V_1|}\chi_{V_1} +\sum_{k=1}^\infty\frac{1}{|V_{k+1}|}\chi_{V_{k+1}\backslash V_k}.$$
Since $T_n^\#$ (and similarly $T_{sup}$) is defined by the formula 
$x\to \int_Gk_n(t)\alpha_{t^{-1}}(x)\,dt$ it will clearly be positivity-preserving on any rearrangement-invariant Banach function space $L^{\bm{\rho}}$ on which it is 
defined. Since $k_n\leq k_{sup}$, that in turn ensures that if $x\geq 0$,  then $T^\#_n(x)=\int_Gk_n(t)\alpha_{t^{-1}}(x)\,dt\leq 
\int_Gk_{sup}(t)\alpha_{t^{-1}}(x)\,dt=T^\#_{sup}(x)$, whence $\|T^\#_n(x)\|_{\bm{\rho}} \leq \|T^\#_{sup}(x)\|_{\bm{\rho}}$. The operators $T^\#_n$ and $T^\#_{sup}$ are of 
course also adjoint preserving. So if $T^\#_{sup}$ is bounded on $L^{\bm{\rho}}$, then we will for any $x\in L^{\bm{\rho}}$ with $x=x^*$ and any $n$ have that 
\begin{eqnarray*}
\|T^\#_n(x)\|_{\bm{\rho}} &\leq& \frac{1}{2}(\|T^\#_n(|x|+x)\|_{\bm{\rho}}+\|T^\#_n(|x|-x)\|_{\bm{\rho}})\\
&\leq& \frac{1}{2}(\|T^\#_{sup}(|x|+x)\|_{\bm{\rho}}+\|T^\#_{sup}(|x|-x)\|_{\bm{\rho}})\\
&\leq& 2\|T^\#_{sup}\|.\|x\|_{\bm{\rho}}.
\end{eqnarray*}
Since for general $x\in L^{\bm{\rho}}$ we have that 
$$\|T^\#_n(x)\|_{\bm{\rho}} \leq \|T^\#_n(\mathrm{Re}(x))\|_{\bm{\rho}} + \|T^\#_n(\mathrm{Im}(x))\|_{\bm{\rho}}$$ 
$$= \|\mathrm{Re}(T^\#_n(x))\|_{\bm{\rho}} + \|\mathrm{Im}(T^\#_n(x))\|_{\bm{\rho}} \leq \|T^\#_n(x)\|_{\bm{\rho}}.$$
it follows that we will then in general have that $\|T^\#_n(x)\|_{\bm{\rho}} \leq 4\|T^\#_{sup}\|_{\bm{\rho}}$.

\section{A Banach principle suited to the present context}\label{S8}

The final task we have is to extract ergodic convergence results from the theory we have developed. We shall in this regard follow the strategy of \cite[5.1 \& 5.3]{dBL}. 
Throughout this subsection one shall assume the operators $T_n:L^{loc}(G)\to C(G)$ ($n\in \mathbb{N}$) to be convolution operators $f\to k_n*f$ induced by non-negative 
functions $k_n$ with $\|k_n\|_\infty\leq 1$ and $\int_Gk_n\,dt=1$. Given a sequence $(S_n)$ of operators from say $L^A(\M,\tau)$ to $M_B(\M,\tau)=M_{\varphi_B}(\M,\tau)$, the 
associated maximal operator $\mathcal{S}$ on $L^A(\M,\tau)$ is defined by $\mathcal{S}:x\mapsto (S_n(x))$. If this operator induces a bounded map 
into $M_{\varphi_B}(\M;\ell^\infty)$, we say that $\mathcal{S}$ is of weak type $(A,B)$. 

To achieve the goal of proving that there is some sort of convergence of $(T_n(f))$ on ``nice'' spaces, our strategy is the following three step programme.
\begin{enumerate} 
\item Given the weak/strong type of the operator $\mathcal{T}$, find the weak/strong type of $T^{\#}$.
\item In the domain of $T^{\#}$ computed in step (1), identify a dense subset $D$ for which some form of convergence of $(T^{\#}_nf)$ can be verified for all $f\in D$.
\item Use an appropriate version of Banach's Principle to extend the convergence of step (2) to the whole domain of $T^{\#}$.\end{enumerate}

We first establisha type of Banach principle appropriate to the context in view.

\begin{proposition}\label{P:osc-alt} Let $X$ and $Y$ be Young functions, and let $(T^\#_n)$ be a sequence of positive operators from say 
$L^X(\M,\tau)$ to $M_Y(\M,\tau)$ induced by convolution operators and let $T^\#_{sup}$ (defined earlier) also be bounded from $L_X(\M,\tau)$ to $M_Y(\M,\tau)$. (Here we are 
looking at weak type. For strong type the range space of the $T_n$'s would be $L^Y(\M,\tau)$.) If

\begin{enumerate}\item $\mathcal{T}$ is of weak-type $(X,Y)$
\item  and there is a dense subset $D\subseteq L^X$ such that for all $f\in D$,  $(T^\#_n(f))$ is Cauchy in the $\sigma(M_Y(\M,\tau),\Lambda_{Y^*}(\M,\tau))$-topology,
\end{enumerate} then is Cauchy in the $\sigma(M_Y(\M,\tau),\Lambda_{Y^*}(\M,\tau))$-topology for all $f\in L^X(\M,\tau)$.\end{proposition}

\begin{proof} Here we make use of the fact that the K\"othe dual of $M(Y)$ is $\Lambda(Y^*)$. See for example the paper https://arxiv.org/abs/1802.01728 on this point. This 
also follows from \cite[Corollary 2.5.14]{BS} and the fact that since $M(\varphi)$ is the largest rearrangement Banach function space with fundamental function $\varphi$, its 
K\"othe dual $M(\varphi)'$ will then be the smallest rearrangement Banach function space with fundamental function $\varphi'$.  

The first thing we need to prove is that $\decra(t,T^\#_n(y))\leq\decra(t,T^\#_{sup}(y))$ for any $t\geq 0$ and any $y\in L^X$, and that it follows from this that 
$\|T^\#_n(y)\|_{M_Y}\leq \|T^\#_{sup}(y)\|_{M_Y}$ for any $y\in L^X$. We leave this as an exercise. Let $u\in D$, $x\in L^X\subset L^{1+\infty}(\M,\tau)$ and 
$v\in \Lambda_{Y^*}(\M,\tau)$ be given. We then have that 
\begin{eqnarray*}
\lim\sup_{n,m}\tau((T^\#_n-T^\#_m)(x),v) &\leq& \lim\sup_{n,m}\tau((T^\#_n-T^\#_m)(x-u)v))+\lim\sup_{n,m}\tau((T^\#_n-T^\#_m)(u)v)\\
&=& \lim\sup_{n,m}\tau((T^\#_n-T^\#_m)(x-u)v)\\
&\leq& \lim\sup_{n,m}\|(T^\#_n-T^\#_m)(x-u)\|_{M_Y}\|v\|_{\Lambda_{Y^*}}\\
&\leq& 8\|T^\#_{sup}(x-u)\|_{M_Y}\|v\|_{\Lambda_{Y^*}}.
\end{eqnarray*}
Here we used the fact that $u\in D$. The claim will therefore follow if we can show that for appropriate choices of $u\in D$, $\|T^\#_{sup}(x-u)\|_{1+\infty}$ can be made 
arbitrarily small. Let $\epsilon>$ be given and select $u\in D$ such that $\|x-u\|_X<\epsilon$. With $\beta$ denoting the norm of $T^\#_{sup}$ on the positive cone of 
$M_Y(\M,\tau)$, we have that $\|T^\#_{sup}(x-u)\|_{M_Y(\M,\tau)}\leq 4\beta\|x-u\|_X \leq 4\beta\epsilon$ which proves the claim.
\end{proof}

\begin{corollary}\label{P:osc-alt2} Let $X$ and $Y$ be Young functions, and let $(T^\#_n)$ be a sequence of positive operators from say $L^X(\M,\tau)$ to $M_Y(\M,\tau)$ and 
define the maximal operator $T^\#_{sup}$ on $L_X(\M,\tau)$ as above. (Here we are looking at weak type. For strong type the range space of the $T^\#_n$'s would be 
$L^Y(\M,\tau)$.) If
\begin{enumerate}\item $T^\#_{sup}$ is of weak-type $(X,Y)$
\item  and there is a dense subset $D\subseteq L^X$ such that for all $f\in D$,  $(T^\#_n(f))$ is Cauchy in the 
$\sigma(L^{1+\infty}(\M,\tau),L^{1\cap\infty}(\M,\tau))$-topology,
\end{enumerate} then $(T^\#_n(f))$ is Cauchy in the $\sigma(L^{1+\infty}(\M,\tau),L^{1\cap\infty}(\M,\tau))$-topology for all $f\in L^X(\M,\tau)$.\end{corollary}

\begin{proof} The proof is much the same as before with the difference that one selects $v$ in $L^{1\cap\infty}$ not $\Lambda_{Y^*}(\M,\tau)$. The final step follows on 
noting that since $M_Y(\M,\tau)$ an exact interpolation space, it embeds into $L^{1+\infty}$. If $C$ is the norm of this embedding on the positive cone, one then has that 
$\|T^\#_{sup}(x-u)\|_{L^{1+\infty}(\M;\ell^\infty)}\leq C\|T^\#_{sup}(x-u)\|_{M_Y(\M;\ell^\infty)}\leq 4C\beta\|x-u\|_X \leq 4C\beta\epsilon$.
\end{proof}

\begin{definition}\label{D:D_X} Let $E_0, E_1$ be subsets of $L^{1+\infty}(\M,\tau)$ and $\mathcal{L}\subseteq L^1(G)$. Define $$D_{E_0}(E_1,\mathcal{L})= 
\{\alpha_f(x): f\in\mathcal{L},~x\in E_0\cap E_1\},$$ which is a subset of $E_0$. (Recall that $\alpha_f(x) = \int_G\alpha_{s}(x)f(s)\,ds$.) In particular, if $\mathcal{F}_0$ 
consists of those integrable functions on $G$ with support of finite measure, we shall simply write $D_{E_0}$ for $D_{E_0}(\M,\mathcal{F}_0)$.\end{definition}

We now verify the validity of step 2 in our proposed scheme.

\begin{proposition}\label{theo2}  Let $(\M,\tau,G,\alpha)$ be a quantum dynamical system and $(T_n)$ a sequence of convolution operators given by $f\mapsto k_n*f$.  Write
$k_n^*$ for the function $t\mapsto \overline{k_n(t^{-1})}\Delta(t^{-1})$. (This is the dual of $k_n$ in $L^1$.) Suppose that the sequence $\big(\int_G k_n(t)~dt\big)$
converges and that $(k_n^**\phi)$ converges weakly in $L^1$ for all $\phi\in L^1(G)$ with support of finite measure.

Then given a noncommutative monotone Riesz-Fischer space $L^{\bm{\rho}}(\M,\tau)$, the sequence $(T^{\#}_nf)$ converges in the 
$\sigma(L^{\bm{\rho}}(\M,\tau), L^{\bm{\rho}'}(\M,\tau))$ topology for every $f\in D_{L^{\bm{\rho}}}$.\end{proposition}

\begin{proof} Firstly note that for any $x\in L^{1+\infty}(\M,\tau)$ we have that $T^{\#}_n(x)=\int_Gk_n(s)\alpha_{s^{-1}}(x)~ds$. (Properly we here have a limit of integrals 
over compact sets increasing to $G$. This follows from the discussion at the start of section \ref{deftran} and the limiting way in which $T^{\#}$ is defined.) Now if 
$x\in D_{L^{\bm{\rho}}}$, then by assumption there must exist a $y\in L^{\infty}(\M,\tau)\cap L^{\bm{\rho}}(\M,\tau)$ and $\psi\in L^1(G)$ with support of finite measure, 
such that $x=\alpha_{\psi}(y)$. Writing $\widehat{k_n}$ for the function $t\mapsto k_n(t^{-1})$, a similar argument to the one used in Proposition \ref{T-posop} but with 
$v\in L^{\bm{\rho}'}(\M,\tau)$ shows that we have that 
\begin{eqnarray*}
T^{\#}_n(x)&=&\int_Gk_n(t)\alpha_{t^{-1}}(x)~dt\\
&=& \int_G k_n(t)\int_G \alpha_{t^{-1}s}(y)\psi(s)~ds~dt\\
&=& \int_G k_n(t)\int_G \alpha_{u}(y)\psi(tu)~du~dt\\
&=& \int_G \alpha_{u}(y)\int_G k_n(t)\psi(tu)~dt~du\\
&=& \int_G \alpha_{u}(y)\int_G k_n(t)\psi(tu)\Delta(t)~d(t^{-1})~du\\
&=& \int_G \alpha_{u}(y)\int_G k_n^*(t^{-1})\psi(tu)~d(t^{-1})~du\\
&=& \int_G \alpha_{u}(y)\int_G k_n^*(v)\psi(v^{-1}u)~dv~du\\
&=& \int_G \alpha_{u}(y)(k_n^**\psi)(u)~du\\
&=& \alpha_{k_n^**\psi}(y)
\end{eqnarray*}
Suppose that $(k_n^**\phi)$ converges weakly in $L^1$ to $g\in L^1(G)$. Let $v\in L^{\bm{\rho}'}(\M,\tau)$ be given. Since $y\in L^{\bm{\rho}}$, we clearly have that 
$|\tau(\alpha_t(y)v)|\leq \|\alpha_t(y)\|_{L^{\bm{\rho}}}\|v\|_{L^{\bm{\rho}'}}= \|y\|_{L^{\bm{\rho}}}\|v\|_{L^{\bm{\rho}'}}$ for any $v\in L^{\bm{\rho}'}(\M,\tau)$. That is 
$t\to \tau(\alpha_g(y),v)$ belongs to $L^\infty(G)$. But then 
$$\tau(T^{\#}_n(x)v)= \int_G \tau(\alpha_{u}(y)v)\int_G k_n^*(v)\psi(v^{-1}u)~dv~du\to \int_G \tau(\alpha_{u}(y)v)g(u)~du=\tau(\alpha_{g}(y)v).$$ That is $(T^{\#}_n(x))$ 
converges to $\alpha_g(y)$ in the $\sigma(L^{\bm{\rho}}(\M,\tau), L^{\bm{\rho}'}(\M,\tau))$ topology.
\end{proof}

\begin{corollary}\label{theo2-b}  Let $(\M,\tau,G,\alpha)$ be a quantum dynamical system and $(T_n)$ a sequence of convolution operators given by $f\mapsto k_n*f$.  Write 
$k_n^*$ for the function $t\mapsto \overline{k_n(t^{-1})}\Delta(t^{-1})$. (This is the dual of $k_n$ in $L^1$.) Suppose that the sequence $\big(\int_G k_n(t)~dt\big)$ 
converges and that $(k_n^**\phi)$ converges weakly in $L^1$ for all $\phi\in L^1(G)$ with support of finite measure.

Given a noncommutative monotone Riesz-Fischer space $L^{\bm{\rho}}(\M,\tau)$, the sequence $(T^{\#}_nf)$ then converges in the 
$\sigma(L^{1+\infty}(\M,\tau), L^{1\cap\infty}(\M,\tau))$ topology for every $f\in D_{L^{\bm{\rho}}}$.\end{corollary}

\begin{proof} On choosing $v$ to be in $L^{1\cap\infty}(\M,\tau)$ rather than $L^{\bm{\rho}'}(\M,\tau)$, much the same proof will suffice to also prove this claim.
\end{proof}

We finally verify the applicability of the strategy to a large class of Orlicz spaces.

\begin{proposition}\label{theoBirk} In the setup of Proposition \ref{theo2}, suppose that $L^X(\M,\tau)$ is an Orlicz space for which the Young function $X$ satisfies 
$\Delta_2$ globally and that the maximal operator $\mathcal{T}^\#$ has weak type $(X,Y)$. Then $(T^{\#}_nf)$ will for every $f\in L^X(\M,\tau)$ converge in the 
$\sigma(M_Y(\M,\tau),\Lambda_{Y^*}(\M,\tau)$-topology.
\end{proposition}

\begin{proof} By hypothesis, $T^\#$ is of weak type $(X,Y)$. To prove the theorem we merely need to show that under the given hypothesis the set $D_{L^X}$ is dense in $L^X$. 
Once this is done we may execute steps (2) and (3) of our three step programme by applying Proposition \ref{P:osc-alt} to the conclusion of  Proposition \ref{theo2}.

It therefore remains to show that $D_X$ is dense in $L^X(\M,\tau)$ when $X\in \Delta_2$. To prove this we first show that all the simple functions with support of finite 
trace belong to the norm closure of $D_X$, after which we note that these simple functions are dense in $X$. So let $p\in \M$ be a projection with $\tau(p)<\infty$. We need 
to show $p$ is the closure of $D_{L^X}$. By \cite[Theorem III.3.2.2]{Bla} the action $t\to\alpha_t(f)$ is point-norm continuous for all $f\in L^1(\M,\tau)$. Hence for $f=p$, 
we will have that $\alpha_t(p)\to p$ in $L^1$-norm as $t\to e$. Next pick a decreasing sequence of compact neighbourhoods $V_m$ in $G$ with intersection $e$ and with 
$\sup_{t\in V_m}\|\alpha_t(p)-p\|_1\leq\frac{1}{m}$. Now let $\psi_m(t)= \frac{1}{|V_m|}\chi_{V_m}$ ($m\in \mathbb{N}$). By definition we then have that 
$(\alpha_{\psi_m}(p))\subset D_X$. If we can show that this sequence convergences to $p$ in the $L^X$-norm, we will have that $p\in D_{L^X}$ as required. 

Firstly note that by \cite[Theorem IV.7.17]{Tak} we will have that  
\begin{eqnarray*}
\|\alpha_{\psi_m}(p)-p\|_1 &=& \left\|\int_G\alpha_t(p)\psi_m(t)\,dt - p\right\|_1\\
&=&\left\|\int_G( \alpha_t(p) - p)\psi_m(t)\,dt\right\|_1\\
&\leq& \int_G\|\alpha_t(p) - p\|_1\psi_m(t)\,dt\\
&\leq& \frac{1}{m}\int_G\psi_m(t)\,dt\\
&=& \frac{1}{m}.
\end{eqnarray*}
The above simultaneously shows that $(\alpha_{\psi_m}(p))\subset L^1(\M,\tau)$ and that $\alpha_{\psi_m}(p)\to p$ in $L^1$. Notice that for any $m$ we will again by 
\cite[Theorem IV.7.17]{Tak} have that 
\begin{eqnarray*}
\|\alpha_{\psi_m}(p)\|_\infty &=& \left\|\int_G\alpha_t(p)\psi_m(t)\,dt\right\|_\infty\\
&=& \frac{1}{|V_m|}\left\|\int_{V_m}\alpha_t(p)\,dt\right\|_\infty\\
&\leq& \frac{1}{|V_m|}\int_{V_m}\|\alpha_t(p)\|_\infty\,dt\\
&=& \|p\|_\infty=1.
\end{eqnarray*}
Thus $(\alpha_{\psi_m}(p))\subset L^{1\cap\infty}(\M,\tau)$. Therefore if also $\|\alpha_{\psi_m}(p)-p\|_\infty\to 0$ as $m\to \infty$, then $\alpha_{\psi_m}(p)\to p$ in 
$L^{1\cap\infty}(\M,\tau)$. Since $L^{1\cap\infty}$ continuously embeds into $L^X$, we will then have that $\alpha_{\psi_m}(p)\to p$ in $L^X$ as required. So suppose that 
$\|\alpha_{\psi_m}(p)-p\|_\infty\not\to 0$. By passing to a subsequence if necessary we may assume that there is some $\epsilon_0>0$ such that 
$\|\alpha_{\psi_m}(p)-p\|_\infty\geq \epsilon_0$ for each $m$. Given $t_0>0$, the convexity of $X$ ensures that $X(t)\leq \frac{X(t_0)}{t_0}t$ for every $0\leq t\leq t_0$. 
For any positive operator $x\in \M$ the Borel functional calculus now ensures that $X(x)\leq \frac{X(\|x\|_\infty)}{\|x\|_\infty}x$. If we apply these facts to the 
$|\alpha_{\psi_m}(p)-p|$s, it follows that $X(|\alpha_{\psi_m}(p)-p|)\leq \frac{X(\|\alpha_{\psi_m}(p)-p\|_\infty)}{\|\alpha_{\psi_m}(p)-p\|_infty}
|\alpha_{\psi_m}(p)-p|\leq \frac{X(2)}{\epsilon_0}|\alpha_{\psi_m}(p)-p|$. But then $\tau(X(|\alpha_{\psi_m}(p)-p|))\leq \frac{X(2)}{\epsilon_0}\|\alpha_{\psi_m}(p)-p\|_1$ 
for all $m$ ensuring that $$\tau(X(|\alpha_{\psi_m}(p)-p|))\to 0\mbox{ as }m\to\infty.$$It follows from for example \cite[Theorem 6.40 \& Corollary 6.42]{GLbook} that 
$\tau(X(|\alpha_{\psi_m}(p)-p|))=\int_0^\infty X(\decra(t,\alpha_{\psi_m}(p)-p))\,dt$ for all $m$. In other words the sequence $(\decra(t,\alpha_{\psi_m}(p)-p))$ of 
$[0,\infty]$-valued functions converge to zero with respect to the convex modulars $\bm{\rho}_X=\int_0^\infty X(\cdot)\,dt$ associated with the Orlicz space 
$L^X(\mathbb{R}_+)$. It is known that modular convergence is equivalent to norm convergence when $X\in \Delta_2$ (see for example \cite[Theorem 3.4.12]{raore}) and so this 
proves that $\|\decra(t,\alpha_{\psi_m}(p)-p)\|_X=\|\alpha_{\psi_m}(p)-p\|_X\to 0$ as $m\to \infty$. This contradicts the assumption that 
$\|\alpha_{\psi_m}(p)-p\|_\infty\not\to 0$, proving that this assumption must be false. We therefore have that $\|\alpha_{\psi_m}(p)-p\|_\infty\to 0$ must hold, which is 
enough to prove that the simple functions with support of finite trace are dense in $L^X$. 
\end{proof} 

\section{Conclusion} 

With the theory developed in this paper at our disposal we now have an effective means of proving ergodic convergence results on noncommutative Orlicz spaces. Assuming $T$ to 
be of weak type $(A,B)$, Theorems \ref{Ttildecont} and \ref{Thashcont} provide us with a mechanism for computing the allowed weak types of $T^\#$ from that of $T$. The manner 
in which that is done is to try and write down a set of Young/fundamental function inequalities that give us the required access to these theorems. So the effectiveness with 
which we can compute the weak type of $T^\#$ is in some sense dependent on our skill in writing down such inequalities. This then matches the status quo in the classical 
setting. By way of example we revisit the example presented in \cite{L-QM}, suitably adapted to the present context. We firstly remind the reader that the Young function for 
the space $L^{1+\infty}(\mathbb{R}_+)$ is given by  
$$Y_{1+\infty}(t)=\left\{\begin{array}{ll} 0 &\quad\mbox{if }0\leq t\leq 1\\ t-1 &\quad\mbox{otherwise}\end{array}\right..$$

The example presented in \cite{L-QM} focused on the space $L\log(L+1)$ generated by the Young function $A_{\log}(t)=t\log(1+t)$ whose significance for quantum physics was 
pointed out in \cite{ML-AHP}. We specifically show how convergence results may be obtained by slightly modifying the function $t\to t\log(t+1)$ to obtain a more regular Young 
function. We specifically modify the behaviour of $A_{\log}$ near 0 to try and ensure more regular behaviour. Instead of $A_{\log}$ we consider the Young function 
$A_{\log+\epsilon}$ given by $$A_{\log+\epsilon}(t)=t\log(1+\epsilon +t)\mbox{ for all }t\geq 0.$$We clearly have that $A_{\log+\epsilon}\geq A_{\log}$. Though more regular 
than $A_{\log}$, this class of Young functions nevertheless do approximate $A_{\log}$ very closely. However as was shown in \cite{L-QM} they have the advantage that they not 
only satisfy $\Delta_2$ globally, but also satisfy the so-called $\Delta'$ condition in the sense that 
$$A_{\log+\epsilon}((\tfrac{\log(1+\epsilon)}{2})st)\leq A_{\log+\epsilon}(t).A_{\log+\epsilon}(s)\mbox{ for all }s,t\geq 0$$ 
and $$(\tfrac{\log(1+\epsilon)}{2})\varphi_{A_{\log+\epsilon}}(st)\leq \varphi_{A_{\log+\epsilon}}(t).\varphi_{A_{\log+\epsilon}}(s) \mbox{ for all }s,t\geq 0.$$Besides the 
above one also has access to the inequality 
$$(\tfrac{\log(1+\epsilon)}{2})\varphi_{\log+\epsilon}(uv)\leq \varphi_{\log+\epsilon}(u).\varphi_{\log+\epsilon}(v)\mbox{ for all }u,v\geq 0$$
Armed with these inequalities one may now conclude from \ref{Thashcont} that any transferable operator $T$ of the type $f\to k*f$ which is of weak(strong) type 
$(Y_{1+\infty},A_{\log+\epsilon})$, yields an operator $T^\#$ of weak(strong) type $(A_{\log+\epsilon},Y_{1+\infty})$. Ergodic convergence results will now follow by simply 
invoking the results recorded in section \ref{S8}.

\end{document}